\documentclass[10pt]{amsart}

\usepackage{amsfonts}
\usepackage{amssymb}
\usepackage{amsmath, amsthm, latexsym, mathbbol, xcolor, mathdots}
\usepackage[all]{xy}
\usepackage[alphabetic]{amsrefs}
 \usepackage{relsize}
\usepackage{hyperref}
\usepackage{graphicx}
 
\usepackage[margin=1in]{geometry}

\def \C{\mathbb{C}}
\def \Z{\mathbb{Z}}
\def \R{\mathbb{R}}

\def \Q{\mathbb{Q}}
\def \A{\mathcal{A}}
\def \P{{\bf P}}
\def \L{{\bf L}}

\def \k{\mathbb{C}}

\def \E{\mathcal{E}}

\def \e{{\bf e}}
\def \f{{\bf f}}
\def \h{{\bf h}}
\def \g{{\bf g}}

\def \H{\mathcal{H}}
\def \M{\mathcal{M}}

\def \MVol{\operatorname{MVol}}

\def \diag{\operatorname{diag}}
\def \ord{\operatorname{ord}}

\def \Gr{\operatorname{Gr}}

\def \supp{\operatorname{supp}}

\def \GL{\operatorname{GL}}

\def \ker{\operatorname{ker}}

\def \conv{\operatorname{conv}}

\def \ev{\operatorname{ev}}

\def \t{\widetilde}

\theoremstyle{plain}
\newtheorem{theorem}{Theorem}[section]
\newtheorem{lemma}[theorem]{Lemma}
\newtheorem{proposition}[theorem]{Proposition}
\newtheorem{corollary}[theorem]{Corollary}

\theoremstyle{definition}

\newtheorem{definition}[theorem]{Definition}
\newtheorem{remark}[theorem]{Remark}

\newtheorem{problem}[theorem]{Problem}

\begin{document}
\title{Vector-valued Laurent polynomial equations, toric vector bundles and matroids}

\author{Kiumars Kaveh}
\address{Department of Mathematics, University of Pittsburgh,
Pittsburgh, PA, USA.}
\email{kaveh@pitt.edu}

\author{Askold Khovanskii}
\address{Department of Mathematics, University of Toronto, Toronto, ON, Canada}
\email{askold@math.toronto.edu}

\author{Hunter Spink}
\address{Department of Mathematics, University of Toronto, Toronto, ON, Canada}
\email{hunter.spink@utoronto.ca}


\thanks{The first author is partially supported by National Science Foundation Grant DMS-210184 and a Simons Collaboration Grant.}
\thanks{The second author is partially supported by the Canadian Grant No. 156833-17}
\thanks{The third author is supported by the Natural Sciences and Engineering Research Council of Canada (NSERC) [RGPIN-2024-04181]}

\begin{abstract}
Let $\L \subset \C^r \otimes \C[x_1^\pm, \ldots, x_n^\pm]$ be a finite dimensional subspace of vector-valued Laurent polynomials invariant under the action of torus $(\C^*)^n$. We study subvarieties in the torus, defined by equations $\f = 0$ for generic $\f \in \L$. We generalize the BKK theorem, that counts the number of solutions of a system of Laurent polynomial equations generic for their Newton polytopes, to this setting. The answer is in terms of mixed volume of certain virtual polytopes encoding discrete invariants of $\L$ which involves matroid data. Moreover, we prove an Alexandrov–Fenchel type inequality for these virtual polytopes. Finally, we extend this inequality to non-representable polymatroids. This extends the usual Alexandrov–Fenchel inequality for polytopes as well as log-concavity results related to matroids.
\end{abstract}

\maketitle

\tableofcontents

\section{Introduction}
We work over the field of complex numbers $\C$, but almost all the results of the paper hold over any
algebraically closed field of characteristic $0$. We let $T=(\k^*)^n$ be the $n$-dimensional torus, and we identify the coordinate ring $\k[T]$ with the ring $\k[x_1^{\pm},\ldots,x_n^{\pm}]$ of Laurent polynomials in $n$ variables. For a finite subset $\A \subset \Z^n$ we denote $$L_\A = \{f(x) = \sum_{\alpha \in \A} c_\alpha x^\alpha \mid c_\alpha \in \k \}=\langle x^\alpha \mid \alpha\in \A\rangle\subset \k[T]$$ for the subspace of Laurent polynomials supported on $\A$. Given finite subsets 
$\A_1, \ldots, \A_r \subset \Z^n$ with $r \leq n$ and $f_i\in \L_{\A_i}$, we denote by $Y(f_1, \ldots, f_r) = \{x \in T \mid f_1(x) = \cdots = f_r(x) = 0\}$ for the subvariety in $T$ defined by $f_1, \ldots, f_r$.

\begin{problem}   \label{problem-1}
For generic $f_i \in L_{\A_i}$, determine the discrete topological and geometric invariants of the subvariety $Y(f_1, \ldots, f_r)$.
\end{problem}

The subject of Newton polyhedron theory is anchored in the surprising fact that the answers to Problem~\ref{problem-1} almost always depend only on the Newton polyhedra $P_i=\conv{\A_i}$ of the Laurent polynomial equations (here $\conv$ denotes the convex hull). For example, the foundational Bernstein-Kushnirenko-Khovanskii (BKK) theorem (\cite{BKK}) computes, when $r=n$, the number of points in $Y(f_1, \ldots, f_n)$ via the mixed volumes of these Newton polyhedra. 
\begin{theorem}[Bernstein-Kushnirenko-Khovanskii]  \label{th-BKK}
For generic $f_i \in L_{\A_i}$, the solutions $x \in T$ of the system:
$$f_1(x) = \cdots = f_n(x) = 0,$$
are all isolated and simple and their number is equal to $n! \MVol_n(P_1, \ldots, P_n)$, where $\MVol_n$ denotes the Minkowski mixed volume of convex bodies in $\R^n$.
Moreover, for any $f_i \in L_{\A_i}$, the number of isolated solutions of the above system, counted with multiplicity, is less than or equal to $n! \MVol_n(P_1, \ldots, P_n)$.    
\end{theorem}

In \cite{Askold-toroidal}, the second author related Problem \ref{problem-1} to the theory of toric varieties by showing that for generic $f_i \in L_{\A_i}$, the closure of $Y(f_1, \ldots, f_r)$ in an appropriate $T$-toric variety $X_\Sigma$ is smooth and transverse to all the orbits. This yields formulae for several discrete topological and geometric invariants of $Y(f_1, \ldots, f_r)$ such as the Euler characteristic, the algebraic and geometric genera, as well as an algorithm for computing the Hodge numbers for the mixed Hodge structure in the cohomology ring of $Y(f_1, \ldots, f_r)$ (\cites{Askold-genus, Danilov-Khovanskii}). These computations have been used in a number of disparate contexts, see for example \cite{Batyrev94} for applications to mirror symmetry.
 


In this paper we extend some of the above results to a much more general class of Laurent polynomial equations. Let $E \cong \k^r$ be an $r$-dimensional vector space over $\k$ where $r \leq n$.
\begin{definition}
    We call a regular map $\Phi:T \to E$ a \emph{vector-valued Laurent polynomial}.
\end{definition}
We identify the vector space of vector-valued Laurent polynomials with $E \otimes \k[T]$. The torus $T$ acts on $E \otimes \k[T]$ by acting on $\k[T]$ in the usual manner. 

We will be interested in $T$-invariant subspaces $\L \subset E \otimes \k[T]$. Every such subspace can be uniquely expressed as follows. Let $\A \subset \Z^n$ be a finite subset of characters of $T$, and for each $\alpha \in \A$, let $E_\alpha \subset E$ be a fixed nonzero linear subspace. Then the arrangement of subspaces $\{E_\alpha\}_{\alpha \in \A}$ defines a $T$-invariant finite dimensional subspace
$$\L = \bigoplus_{\alpha \in \A} E_\alpha \otimes x^\alpha\subset E\otimes \k[T].$$
For generic $\f \in \L$, we are interested in the subvariety $Y(\f)$ consisting of the solutions $x \in T$ of the vector equation $\f(x) = 0$, that is:
\begin{align*} 
Y(\f) = \{x \in T \mid \f(x) = \sum_{\alpha \in \A} e_{\alpha} x^\alpha = 0\},
\end{align*}
where the $e_{\alpha} \in E_{\alpha}$ are chosen generically. Because $Y(\f)$ only depends on $\sum_{\alpha\in \mathcal{A}} E_\alpha$, from now on we will assume without loss of generality that \begin{align}
\tag{$\ast$}E=\sum_{\alpha\in \mathcal{A}}E_\alpha.\end{align}
We will write $r=\dim E=\dim \sum E_\alpha$. Generalizing the classical Newton polyhedra theory we pose the following problem.
\begin{problem}    \label{problem-2}
For generic $\f \in \L$, determine the discrete topological and geometric invariants of the subvariety $Y(\f)$ in terms of the discrete/combinatorial invariants of the subspace $\L$. 
\end{problem}

\begin{remark}   \label{rem-Esterov}
When the first version of the present paper was completed, we learned about the interesting recent series of papers by Alexander Esterov \cite{E1, E2, E3} where the author considers basically the same setup as in our paper, namely, subvarieties $Y(\f)$ defined by a vector-valued Laurent polynomial $\f$, where $\f$ is generic in a torus invariant subspace $\L$. In these papers, $Y(\f)$ is referred to as an \emph{engineered complete intersection}. The author computes some of the important discrete topological and geometric invariants of $Y(\f)$.    
\end{remark}

The special case of Problem~\ref{problem-2} when all of the $E_\alpha$ are coordinate subspaces of $E$ (that is, spanned by subsets of a fixed basis of $E$) recovers Problem \ref{problem-1}. Indeed, given a system of generic Laurent polynomial equations $f_1=\cdots=f_r=0$, we take the coordinate subspace arrangement $\{E_\alpha \mid \alpha\in \bigcup_{i=1}^r \mathcal{A}_i\}$ where $E_\alpha\subset E$ is the span of the basis vectors indexed by $\{i \mid \alpha\in \mathcal{A}_i\}$.

In Section \ref{subsec-L-non-degen} we will quantify the genericity we take $\f$ with respect to, through a new notion of \emph{$\L$-non-degeneracy}. This extends the notion of $\Delta$-non-degeneracy from \cite{Askold-toroidal}, and the locus of $\L$-non-degenerate $\f \in \L$ is a non-empty Zariski open subset. From standard arguments in differential topology, the subvarieties $Y(\f)$ for $\L$-non-degenerate $\f$ are all diffeomorphic (Theorem \ref{th-non-degenerate-diffeo}) and hence possess the same topological invariants. 
We also expect that many of their natural discrete algebro-geometric invariants are also the same, but we do not pursue this here.

\subsection{Discrete combinatorial data}
We now describe the discrete invariants of $\L$ on which our results will depend on. The most concrete way to understand our data will be through a sequence of lattice polytopes $(\Delta_1, \ldots, \Delta_r)$. We think of the virtual polytopes $\Delta_1,\Delta_2-\Delta_1,\ldots,\Delta_r-\Delta_{r-1}$ as an analogue of the sequence of Newton polytopes $P_i$ in Problem \ref{problem-1}.
\begin{definition}   \label{def-Delta-k}
For $1 \leq i \leq r$, say that an $i$-tuple $(\alpha_1, \ldots, \alpha_i) \in \A^i$ is \emph{admissible} for $\L$ if we can pick $e_j \in E_{\alpha_j}$, $j=1, \ldots, i$, in such a way that $e_1,\ldots,e_i$ are linearly independent.
Define the polytope $\Delta_i$ by:
\begin{align*}  
\Delta_i = \conv\{ \sum_{j=1}^i \alpha_{j} \mid (\alpha_1, \ldots, \alpha_i) \in \A^i \textup{ is admissible for } \L \}. 
\end{align*}
In particular, $\Delta_1$ is the convex hull of $\A$. We also set $\Delta_0 = \{0\}$. We call $(\Delta_1, \ldots, \Delta_r)$ the \emph{characteristic sequence of polytopes} of $\L$.
\end{definition}

On the other hand, the subspace $\L$ determines a \emph{multi-valued support function} $\h_\L$ on which our main results primarily depend on. It can be recovered from the polytopes $\Delta_i$. The multi-valued support function $\h_\L$ arises as the equivariant Chern roots of a toric vector bundle associated to $\L$ which controls the geometry of $Y(\f)$ (see Section \ref{sec-tvb}).


For a dual vector $\xi \in (\R^n)^*$ and $c \in \R$, define the subspace $E^\xi_c$ by:
\begin{equation} \label{equ-intro-E-xi}
E^\xi_c = \sum_{\langle \xi, \alpha \rangle \leq c} E_\alpha.
\end{equation}
For each $\xi$, $\{E^\xi_c\}_{c \in \R}$ is an increasing filtration in $E$. We say that $c$ is a \emph{critical number} of multiplicity $d>0$ for $\xi$, if the quotient space $E^\xi_c / \sum_{c' < c} E^\xi_{c'}$ has dimension $d$. For fixed $\xi$, the sum of the multiplicities of all the critical numbers $c$ is equal to $r = \dim(E)$.

\begin{definition}[Multi-valued support function]   \label{def-multi-supp}
Let $\L$ be an invariant subspace with rank $r$. The \emph{$r$-valued support function} of $\L$ is the function $\h_\L$ that sends a dual vector $\xi \in (\R^n)^*$ to the multiset of critical numbers of $\xi$, where each critical number $c_i$ is repeated as many times as its multiplicity $d_i$.  
\end{definition}
The function $\h_\L$ is piecewise linear in the following sense. There exists a (not necessarily unique) collection $\tilde{\h} = 
\{h_1, \ldots, h_r\}$ of $r$ piecewise linear functions which \emph{represents} $\h_\L$. This means that for any $\xi$, the multiset $\h_\L(\xi)$ coincides with $\{h_1(\xi), \ldots, h_r(\xi)\}$. In other words, for any critical number $c$, its multiplicity $d$ is equal to the number of indexes $j$ such that $h_j(\xi) = c$. We would like to emphasize that an $r$-valued support function $\h$ can be represented by different collections $\tilde{\h}$. Nevertheless, many quantities associated to different collections $\tilde{\h}$ representing the same $r$-valued support function $\h$ are the same.

\begin{definition} \label{def-MV(h)}
Let $r=n$ and let $\h$ be an $n$-valued support function and pick a representation of it by $n$ piecewise linear functions $\{h_1, \ldots, h_n\}$. We define the \emph{mixed volume} $\MVol(\h)$ to be $\MVol(P_1, \ldots, P_n)$, where $P_i$ is the virtual polytope corresponding to the piecewise linear functions $h_i$.
\end{definition}
One shows that this mixed volume is well-defined, that is, only depends on the $n$-valued function $\h$ represented by the $h_i$ and not on the particular representation $\{h_1, \ldots, h_n\}$ (see Section \ref{sec-convex-chains}).

\begin{remark}
The notion of a multi-valued support function above is a particular case of a more general notion of a mutli-valued support function in \cite{KhPu} where the values of the support function can have any integer multiplicities (as opposed to having only positive integer multiplicities). These multi-valued support functions give an extension of the notion of support function of a convex polytope to \emph{convex chains}. This is also related to the notion \emph{branched cover of a fan} in \cite{Payne-cover}.     
\end{remark}

\begin{theorem}    \label{th-intro-multi-supp-char-poly}
As above, let $\L$ be an invariant subspace of rank $r$, and let $\{h_1,\ldots,h_r\}$ be the representation of $\h_\L$ such that for any $\xi$, 
$h_1(\xi) \leq \cdots \leq h_r(\xi)$. Then for $i=1, \ldots, r$, we have:
$$h_i = h_{\Delta_i} - h_{\Delta_{i-1}},$$
where $h_{\Delta_i}(\xi)=\min_{v\in \Delta_i} \langle \xi,v\rangle$ is the support function of the characteristic polytope $\Delta_i$. 
\end{theorem}
\subsection{Main results}
Let $X_\Sigma$ be a smooth $T$-toric variety whose fan $\Sigma$ is a subdivision of the normal fan of $\Delta_1 + \cdots + \Delta_r$. The following is our first result, extending a corresponding result from \cite{Askold-toroidal}. It is proved in Section \ref{subsec-L-non-degen}.
\begin{theorem}   \label{th-intro-transversality}
Let $\L$ be an invariant subspace of vector-valued Laurent polynomials. For generic $\f \in \L$, the closure of $Y(\f)$ in the toric variety $X_\Sigma$ is smooth and transverse to all the orbits in $X_\Sigma$.
\end{theorem}


Associated to $\L$ and a fan $\Sigma$ refining the normal fan of $\Delta_1+\cdots+\Delta_r$, we will associate a $T$-equivariant vector bundle $\E_{\L,\Sigma}$ such that $\L$ naturally lies inside $H^0(X_\Sigma, \E_{\L,\Sigma})$ as a $T$-representation. We will denote the section of $\E_{\L,\Sigma}$ associated to $\f$ by $s_\f$. Also, we define the decreasing analogue of the filtration $(E^\xi_c)_{c \in \R}$ by $\t{E}^\xi_c = \sum_{\langle \xi, \alpha \rangle \geq c} E_\alpha$. The following theorem and corollary are shown in Section \ref{sec-tvb}.

\begin{theorem}   \label{th-intro-tvb}
With setup as above, the following are true.
\begin{itemize}
\item[(a)] The (decreasing) filtrations $\{\t{E}^\xi_c\}_{c\in \Z}$, where $\xi$ runs over the primitive generators of the rays in $\Sigma$, are the Klyachko filtrations determining the toric vector bundle $\E_{\L,\Sigma}$. Moreover, $\h_\L$ coincides with the multi-valued support function determined by the characters in Klyachko's compatibility conditions of the toric vector bundle $\E_{\L, \Sigma}$ (representing its equivariant Chern roots).  
\item[(b)] The zero locus of $s_\f$ in the open orbit $T\subset X_\Sigma$ is the solution set $Y(\f)$.
\item[(c)] For any $x\in X_\Sigma$, the composite $\L\subset H^0(X_\Sigma, \mathcal{E}_{\L,\Sigma})\to \mathcal{E}_{\L,\Sigma}|_x$ is surjective.
\end{itemize}
\end{theorem}
Statement (a) follows from Klyachko’s classification of toric vector bundles (\cite[\S 2]{Klyachko}). 
Statements (b) and (c) are proved by the same arguments as in the proof of Theorem \ref{th-intro-transversality}. 

\begin{corollary} The equivariant Chern roots of the toric vector bundle $\E_{\L, \Sigma}$ from Theorem \ref{th-intro-tvb} are represented  by any collection of functions $\{h_1, . . . , h_r\}$ representing the multi-valued function $\h_\L$.
\end{corollary}

The following extends the BKK theorem (Theorem~\ref{th-BKK}) to our setting (see Section \ref{sec-BKK}).
\begin{theorem}    \label{th-main}
Let $\L$ be an invariant subspace of rank $r=n$. Then, for generic $\f \in \L$, all the points in $Y(\f)$ are isolated and non-degenerate and the number of points in $Y(\f)$ is given by:
$$|Y(\f)| = n! \MVol_n(\h_\L).$$
In particular, this number can be represented in terms of the mixed volume of the virtual polytopes $\Delta_i - \Delta_{i-1}$, namely:
$$|Y(\f)| = n! \MVol_n(\Delta_1,\Delta_2 - \Delta_1, \ldots, \Delta_n - \Delta_{n-1}).$$
Moreover, for any $\f \in \L$, the number of isolated solutions, counted with multiplicity, is $\leq n! \MVol_n(\h_\L)$.
\end{theorem}

The \emph{ring of conditions} of a homogeneous space was introduced in \cite{DeConcini-Procesi}. For the torus $T$ the ring of conditions can be viewed as the natural ring to do intersection theory of subvarieties in $T$, and is a direct limit of the (Chow) cohomology rings of all toric compactifications of $T$. 
The ring of conditions has a tropical/combintatorial description as the ring of Minkowski weights (see \cite{EKKh}, as well as \cite{KKh-survey} for a brief survey, this is also related to \cite{Fulton-Sturmfels}).

The ring of conditions for the torus $T$ has generators $[P]$ for every lattice polytope $P$ representing the hypersurface $f=0$ for $f$ a Laurent polynomial with Newton polytope $P$. Because of this fact, the following theorem is an immediate corollary of Theorem \ref{th-main} (also see Section \ref{sec-BKK}).
\begin{theorem}  \label{th-intro-ring-of-conditions}
Let $r \leq n$ and let $\L$ be a rank $r$ invariant subspace of vector-valued Laurent polynomials. Let $\{h_1, \ldots, h_r\}$ be a representation of the $r$-valued support function $\h_\L$ and let $P_i$ be the virtual polytope corresponding to $h_i$. Then for generic $\f \in \L$ the class of the subvariety $Y(\f)$ in the ring of conditions is given by the product:
$$[Y(\f)] = \prod_{i=1}^r [P_i].$$
In particular, in terms of the characteristic polytopes of $\L$, we have: 
$$[Y(\f)] = \prod_{i=1}^r [\Delta_i - \Delta_{i-1}].$$
\end{theorem}


Another corollary of Theorem \ref{th-main} is the following result (Section \ref{sec-hyperplane-arrangement}), which is equivalent to a computation known in the tropical geometry literature for the class of a linear space in the ring of conditions of $T$ as a product of virtual polytopes associated to matroids (see June Huh's thesis \cite{Huh} and further discussions in \cite[Appendix III]{BEST}).
To state this result, let $\H = \{H_0, \ldots, H_N\}$ be a hyperplane arrangement in the projective space $\mathbb{P}^n$, where $H_i$ is the (projective) hyperplane defined by the linear equations $u_i = 0$. For $i=1, \ldots, n$ define the polytope $\Delta_i \subset \R^{N+1}$ by:
$$\Delta_i = \conv\{e_{j_1} + \cdots + e_{j_i} \mid \bigcap_{j \notin \{j_1, \ldots, j_i\}} H_j = \emptyset  \}.$$ These are the matroid polytopes of the truncations of the dual matroid associated to $\mathcal{H}$.

\begin{corollary}[BKK with homogeneous linear coordinates]
Let $f_i \in L_{\A_i}$, $i=1, \ldots, n$, be generic homogeneous degree $0$ Laurent polynomials in $z=(z_0, \ldots, z_N)$ with fixed supports $\A_i \subset \Z^{N+1}$ that lie in the hyperplane given by the sum of coordinates equal to $0$. Then the number of solutions $x \in X = \mathbb{P}^n \setminus \bigcup_{i=0}^N H_i$ of  the system $f_1(u_0(x), \ldots, u_N(x)) = \cdots = f_n(u_0(x), \ldots, u_N(x)) = 0$ is equal to:
$$N! \MVol(\Delta_1, \Delta_2 - \Delta_1, \ldots, \Delta_{N-n} - \Delta_{N-n-1}, P_1, \ldots, P_n),$$
where $P_i = \conv(\A_i)$, $i=1, \ldots, n$. Note that strictly speaking, the polytopes $P_i$ and $\Delta_i$ do not lie in the same hyperplane, but rather they all lie in parallel hyperplanes defined by the sum of coordinates equal to constants. The mixed volume above, denotes the mixed volume in the hyperplane the sum of coordinates equal to $0$, and after translating the polytopes to this hyperplane.
\end{corollary}

The classical Alexandrov–Fenchel inequalities for mixed volumes of lattice polytopes $C_1,\ldots,C_{n-2},P,Q\subset \mathbb{R}^n$ states that
$$\MVol_n(C_1,\ldots,C_{n-2},P,P)\MVol_n(C_1,\ldots,C_{n-2},Q,Q)\le \MVol_n(C_1,\ldots,C_{n-2},P,Q)^2.$$
These inequalities can be translated to counts of solutions to systems of linear equations via the BKK theorem, in which case it becomes an instantiation of the Khovanskii-Teissier inequality on a good compactification of $T$. 
We now give a generalization of the Alexandrov–Fenchel inequality to the setting of vector-valued Laurent polynomial equations.
Let $\L_i$, $i=1, 2$, be invariant subspaces defined by arrangement of subspaces $\{E_{i, \alpha}\}_{\alpha \in \A}$ in the vector spaces $E_i$. We define the direct sum $\L_1 \oplus \L_2$ to be the invariant subspace corresponding to the arrangement $\{E_{1, \alpha} \oplus E_{2, \alpha}\}_{\alpha \in \A}$ in $E_1 \oplus E_2$. The following is proven in Section \ref{sec-AF}.

\begin{theorem} \label{th-intro-AF}
Let $\L_1$, $\L_2$ be invariant subspaces of rank $1$ and $\L_3$ an invariant subspace of rank $n-2$. We then have:
\begin{align*}
\MVol(\h_{\L_1 \oplus \L_2 \oplus \L_3})^2 \geq \MVol(\h_{\L_1 \oplus \L_1 \oplus \L_3}) \MVol(\h_{\L_2 \oplus \L_2 \oplus \L_3}).
\end{align*}
\end{theorem}
When $\L_3$ is a direct sum of invariant subspaces of rank $1$ this exactly recovers the translation of the Alexandrov-Fenchel inequality to algebraic geometry via the BKK theorem. We will prove Theorem \ref{th-intro-AF} by a more involved application of the Khovanskii-Teissier inequalities to certain intermediate algebraic varieties which arise in the course of the paper. By taking $\L_3$ to be the direct sum of a fixed $\L$ with rank $1$ linear spaces, we obtain the following corollary.
\begin{corollary}   \label{cor-intro-AF-v2} 
Let $\L$ be an invariant subspace of rank $r$ where $r \leq n$. Let $(\Delta_1, \ldots, \Delta_r)$ be the characteristic sequence of polytopes associated to $\L$. Also let $P_1, \ldots, P_{n-r}$ be arbitrary convex bodies in $\R^n$. Then we have:
$$\MVol_n(\Delta_1, \ldots, \Delta_{r}-\Delta_{r-1}, P_1, P_2, P_3, \ldots, P_{n-r}))^2 \geq $$
$$\MVol_n(\Delta_1, \ldots, \Delta_{r}-\Delta_{r-1}, P_1, P_1, P_3, \ldots, P_{n-r}) \MVol_n(\Delta_1, \ldots, \Delta_{r}-\Delta_{r-1}, P_2, P_2, P_3, \ldots, P_{n-r})).$$     
\end{corollary}


Finally, in Section \ref{sec-polymatroid}, we extend our Alexandrov-Fenchel inequality to a combinatorial abstraction of the subspace arrangements $\{E_\alpha \mid \alpha\in \A\}$ such that $\L=\bigoplus E_\alpha\otimes x^\alpha$ known as a \emph{polymatroid}, which we recall formally in Section~\ref{sec-polymatroid}. Concretely, a polymatroid $\P$ is given by a ``dimension function'' from subsets of $\mathcal{A}$ to $\{0,1,2,\ldots\}$ which satisfies certain basic axioms that always hold for the function $S\mapsto \dim \sum_{\alpha \in S} E_\alpha$. It turns out that the admissible sequences can be defined from this function, and so one may define analogues of the characteristic polytopes $\Delta_{i,\P}$ and multi-valued support function $h_\P$.

Hodge theory for matroids, as developed in \cite{AHK} and \cite{ADH}, shows in certain rather special situations we can extend log-concavity results such as the Alexandrov-Fenchel inequality to the more general matroid setting. These situations all involve intersecting the class of a linear space closure in a torus with systems of generic Laurent polynomial equations.

\begin{theorem}[An Alexandrov-Fenchel inequality involving arbitrary polymatroids]  \label{th-intro-AF-matroid}
The inequality in Theorem \ref{th-intro-AF} holds for arbitrary polymatroids $\P$, and all three terms appearing in the inequality are nonnegative integers.  
\end{theorem}

Our proof of Theorem \ref{th-intro-AF-matroid} relies on Hodge theory of matroids (as in \cite{AHK}) and generalizes the method used in \cite{BEST} to derive an algebraic framework to understand matroid computations by showing that the closure of a linear space $L\cap T$ in a compactification of $T$ can be viewed as a top Segre class $c_{\text{top}}(\mathcal{S}_L)$ of a torus-equivariant vector bundle $\mathcal{S}_L$. From this one can deduce Hodge-theoretic inequalities by doing tropical intersection theory on the projectivization $\mathbb{P}(\mathcal{S}_L)$. Subsequent work in this direction was \cite{EFLS} where similar techniques were used to obtain inequalities for delta matroids, and \cite{EHL} where affine spaces were modeled as Chern classes of vector bundles, by viewing $L$ as a subspace of particular $T$-representations adapted to the problems at hand. Our result essentially follows by considering linear spaces inside arbitrary representations of $T$, and the associated Segre classes then compute the vector-valued systems of Laurent polynomial equations $Y(\f)$.

\begin{remark}
In \cite{KM-TMB} and \cite{KhM} a combinatorial/tropical generalization of the notion of a toric vector bundle is introduced for matroids called a \emph{tropical vector bundle}. One can show that, the above data of a polymatroid in fact gives rise to a tropical vector bundle.   
\end{remark}


\bigskip

\noindent{\bf Notation:}
\begin{itemize}
    \item $T = (\k^*)^n$, $n$-dimensional algebraic torus with character lattice $M \cong \Z^n$ and lattice of one-parameter subgroups $N \cong \Z^n$.
    We denote the natural pairing $N \times M \to \Z$ by $\langle \cdot, \cdot \rangle$. The coordinate ring of $T$ is $\k[T]$, a Laurent polynomial ring in $n$ variables. For $\alpha \in M$, we denote the corresponding character by $x^\alpha$.
    \item $E$, an $r$-dimensional $\k$-vector space with $r \leq n$.
    \item $\A$, a finite subset of characters of $T$.
    \item $\L$, a finite dimensional $T$-invariant subspace of $E \otimes \k[T]$. We refer to elements of $E \otimes \k[T]$ as vector-valued Laurent polynomials.
    \item $\{E_\alpha\}_{\alpha \in \A}$, an arrangement of subspaces $E_\alpha \subset E$ labeled by elements $\alpha \in \A$. This data determines an invariant subspace $\L = \bigoplus_{\alpha \in \A} E_\alpha \otimes x^\alpha$ and every invariant subspace is of this form. Without loss of generality we always assume that $\sum_{\alpha \in \A} E_\alpha = E$.
    \item $Y(\f)$ is the subvariety in $T$ defined by a vector-valued equation $\f(x)=0$ for $\f \in \L \subset E \otimes \k[T]$.
    \item $\h_\L$, the multi-valued support function associated to $\L$.
    \item $(\Delta_1, \ldots, \Delta_r)$, the characteristic sequence of polytopes of an invariant subspace $\L \subset E \otimes \k[T]$.
    \item $\Sigma_\L$, the complete fan associated to an invariant subspace $\L$ defined as the normal fan of the Minkowski sum $\sum_{i=1}^r \Delta_i$.
    \item $\E_{\L, \Sigma}$, the toric vector bundle on the toric variety $X_{\Sigma}$ associated to $\L$, where $\Sigma$ is a fan refining $\Sigma_\L$.
\end{itemize}

\section{Mixed volume of an $n$-valued support function}   \label{sec-convex-chains}
In this section we review the notion of a multi-valued support function (\cite{KhPu}) and recall how a formula of Brion \cite[Corollary in Section 2.4]{Brion} shows that the mixed volume of an $n$-valued support function is well-defined (Definition \ref{def-MV(h)}).

The collection of convex polytopes in $M_\R$ is a cancellative semigroup with respect to Minkowski sum. 
Moreover, one can scale convex polytopes by nonnegative real numbers. Thus, one can extend the collection of convex polytopes to the $\R$-vector space of \emph{virtual polytopes} by considering formal differences of polytopes. 

We recall that given a convex polytope $P \subset M_\R \cong \R^n$, its \emph{support function} is the function $h_P: N_\R \to \R$ defined by:
$$h_P(\xi) = \min\{\langle \xi, \alpha \rangle \mid \alpha \in P\}, \quad \forall \xi \in N_\R.$$
\begin{remark}  \label{rem-supp-function-max-min}
Some authors define the the support function using a maximum convention, that is, $\t{h}_P(\xi) = \max\{\langle \xi, \alpha \rangle \mid \alpha \in P\}$. The two definitions are related by $\t{h}_P(\xi) = -h_P(-\xi)$, for all $\xi \in N_\R$.
\end{remark}

The notion of support function can be extended to virtual polytopes by defining the support function $h_{P_1 - P_2}$ of a virtual polytope $P_1 - P_2$ to be $$h_{P_1 - P_2} = h_{P_1} - h_{P_2}.$$
It is well-known that $P \mapsto h_P$ gives an isomorphism between the $\R$-vector space of virtual polytopes in $M_\R$ and the $\R$-vector space of piecewise linear functions on $N_\R$.

\begin{definition}[Multi-valued support function]
A \emph{multi-valued support function} on $N_\R$ is a function $\h$ from $N_\R$ to the group ring $\Z[\R]$ that can be represented as:
$$\h = \sum_i m_i h_i,$$
where the $m_i \in \Z$ and the $h_i(\xi)$ are (single valued) piecewise linear functions on $N_\R$.
In this paper, if the $m_i$ are nonnegative we refer to $\h$ as an \emph{effective} multi-valued support function. In this case we also say $\h$ is an $r$-valued support function where $r = \sum_i m_i$.
\end{definition}

Roughly speaking, an effective multi-valued support function assigns to $\xi \in N_\R$, the multi-set of values $h_i(\xi)$ where each $h_i(\xi)$ is repeated $m_i$ times. 

It will be useful to define a canonical representation of an effective multi-valued support function into piecewise-linear functions (although we stress that most of the results in this paper do not depend on which representation is taken).
\begin{definition}
    Given an effective multi-valued support function $\h$, we define the \emph{canonical representation} of $\h$ to be the unique collection of piecewise-linear functions $h_1,h_2,\ldots,h_r$ such that $\h(\xi)=\{h_1(\xi),\ldots,h_r(\xi)\}$ as multisets, and $h_1(\xi)\le \cdots \le h_r(\xi)$ for any $\xi \in N_\R$.
\end{definition}
\begin{remark}
     Recall that a convex chain is a linear combination $\sum_i m_i \mathbb{1}_{P_i}$ of indicator functions of convex polytopes $P_i \subset M_\R$ with $m_i \in \Z$. It can be shown that $C \mapsto \sum_i m_i h_{P_i}$ gives a well-defined bijection between convex chains in $M_\R$ and multi-valued support functions on $N_\R$ (see \cite[Section 4]{KhPu}). Note that $C$ can be represented as $\sum_i m_i \mathbb{1}_{P_i}$ in many different ways, but for all such representations the multi-valued function $\sum_i m_i h_{P_i}$ is the same.
\end{remark}
 



Let $\h$ be an effective $n$-valued support function and let us write $\h = \sum_{i=1}^n h_i$ where the $h_i$ are (single-valued) support functions on $N_\R$.
Recall (Definition \ref{def-MV(h)}) that we define the mixed volume $\MVol(\h)$ to be: $$\MVol_n(\h) = \MVol_n(P_1, \ldots, P_n),$$ where $P_i$ is the virtual polytope whose support function is $h_i$.

\begin{theorem} \label{th-mvol-multi-valued-well-def}
The definition of $\MVol(\h)$ is well-defined, that is, is independent of the representation of $\h$ as a multiset of piecewise linear functions $\{h_1, \ldots, h_n\}$.  
\end{theorem}
\begin{proof}
One knows that the mixed volume of the virtual poytopes $P_1, \ldots, P_n$ only depends on the product $h_1 \cdots h_n$, which is a piecewise polynomial function (see \cite[Corollary in Section 2.4]{Brion}). The claim immediately follows from this fact. 
\end{proof}

The theory of convex chains allows one to assign a convex chain to any given symmetric polynomial with integral coefficients in the values of effective multi-valued support functions. Thus, chains (and their invariants) depend only on the effective multi-valued function and symmetric functions on its values. The mixed volume of an $n$-valued effective support function is among such invariants of effective multi-valued functions. Other invariants of such kind appear in the other particular cases of our Problem \ref{problem-2}.

\section{Discrete invariants of a $T$-invariant subspace $\L$}
Let $E \cong \k^r$ be an $r$-dimensional vector space over $\k$ where $r \leq n$. We are interested in the vector space $E \otimes \k[T]$. We refer to elements of $E \otimes \k[T]$ as \emph{vector-valued Laurent polynomials of rank $r$}. The torus $T$ acts on $E \otimes \k[T]$ by acting on $k[T]$ in the usual manner.

Let $\A$ be a finite set of characters of $T$.  For each $\alpha \in \A$, let $E_\alpha \subset E$ be a non-trivial linear subspace. Without loss of generality we assume: $$\sum_{\alpha \in \A} E_\alpha = E.$$
Then the subspace $\L \subset E \otimes \k[T]$ defined by:
\begin{equation}  \label{equ-L}
\L = \bigoplus_{\alpha \in \A} E_\alpha \otimes x^\alpha \subset E,
\end{equation}
is a $T$-invariant subspace of vector-valued Laurent polynomials. Conversely, any $T$-invariant subspace is of the form \eqref{equ-L}. Thus, the data $\{E_\alpha\}_{\alpha \in \A}$ and an invariant subspace $\L$ are the same set of information. 

For $\f \in \L$, we let $Y(\f) \subset T$ to be the subvariety in the torus defined by $\f(x) = 0$. That is:
\begin{equation} 
Y(\f) = \{x \in T \mid \f(x) = \sum_{\alpha \in \A} e_{\alpha} x^\alpha = 0\}.
\end{equation}

\subsection{Multi-valued piecewise linear function associated to an invariant subspace}
\label{subsec-multi-val-supp}
To an invariant subspace $\L \subset E\otimes \k[T]$, we associate an $r$-valued support function $\h_\L$ (see Section \ref{sec-convex-chains}) as follows. It plays a central role in the paper.


\begin{definition}[Filtrations associated to $\L$] \label{def-inc-filt} 
An element $\xi \in N_\R$ gives us an increasing $\R$-filtration $(E^\xi_i)_{i \in \R}$ in the space $E$ by:
\begin{equation} \label{equ-inc-filt}
E^\xi_i = \sum_{\alpha \in \A,~ \langle \xi, \alpha \rangle \leq i} E_\alpha.   \end{equation} 
\end{definition}

\begin{remark}
\label{rem-klyachkoconvention}
The filtrations $(E^\xi_i)_{i \in \R}$ are analogues of the filtrations appearing in Klyachko's classification of torus equivariant vector bundles (\cite{Klyachko}). Our convention here is opposite to that of the Klyachko and hence we have increasing filtrations (instead of decreasing filtrations in \cite{Klyachko}). See Sections \ref{sec-tvb}.   
\end{remark}

We say that $c$ is a \emph{critical number} of multiplicity $d>0$ for $\xi$, if the quotient space $E^\xi_c / \sum_{c' < c} E^\xi_{c'}$ has dimension $d$. For fixed $\xi$, the sum of the multiplicities of all the critical numbers $c$ is equal to $r = \dim(E)$.

\begin{definition}[Multi-valued support function of $\L$]  \label{def-multi-val-supp-function}
The \emph{$r$-valued support function} of $\L$ is the function $\h_\L$ that sends $\xi \in N_\R$ to the multi-set of critical numbers of $\xi$, where each critical number $c_i$ is repeated as many times as its multiplicity $d_i$.
\end{definition}

\begin{definition}[Partial flags associated to $\L$]   \label{def-parital-flag-L}
The filtration $(E^\xi_i)_{i \in \R}$ is the same information as a flag of subspaces:
$$F^\xi_\bullet = (\{0\}=F^\xi_0 \subsetneqq F^\xi_1 \subsetneqq \cdots \subsetneqq F^\xi_k = E),$$
labeled by a strictly increasing sequence of numbers $c^\xi_\bullet = (c^\xi_1 < \cdots < c^\xi_k)$. For ease of notation, we often drop $\xi$ and write $F_j$ and $c_j$ in place of $F^\xi_j$ and $c^\xi_j$ respectively.
The $F_j$ are the subspaces appearing in the filtration and $c_j$, $j=1, \ldots, k$, is the first position in the filtration where the subspace $F_j$ appears. That is, $c_j$ is the minimum number such that $F_j = E^\xi_{c_j}$, $j=1, \ldots, k$. We let $d_j = \dim(F_j)$. Clearly, $0=d_0 < d_1 < \cdots < d_k = r$. In other words, $c_j$ is the \emph{$j$-th critical number of} the filtration with multiplicity $d_j-d_{j-1}$.    
\end{definition}

 The canonical representation of $\h_\L$ in terms of piecewise-linear functions $h_1\le h_2\le \cdots \le h_r$ is constructed as follows: for $\xi \in N_\R$, define the values of $h_1(\xi), \ldots, h_r(\xi)$ by
\begin{equation}  \label{equ-c-i}
h_1(\xi) = \cdots = h_{d_1}(\xi) = c_1, 
\end{equation}
$$h_{d_1+1}(\xi) = \cdots = h_{d_2}(\xi) = c_2,$$
$$\cdots$$
$$h_{d_{k-1}+1}(\xi) = \cdots = h_r(\xi) = c_k.$$

We end this section by an observation about the multi-valued support function of a direct sum. Let $E_1$, $E_2$, be vector spaces of dimensions $r_1$, $r_2$ respectively. Let $E=E_1 \oplus E_2$ with $\dim(E)= r = r_1 + r_2$. Let $\L_1, \L_2$ be invariant subspaces in $E_1 \otimes \k[T]$, $E_2 \otimes \k[T]$ respectively. From the construction, one verifies the following. 
\begin{proposition}[Support function of a direct sum] \label{prop-supp-function-direct-sum}
The multi-valued support function $\h_\L$ corresponding to $\L = \L_1 \oplus \L_2 \subset E \otimes \k[T]$ is given by merging the support functions $\h_{\L_1}$, $\h_{\L_2}$. More precisely, for $\xi \in N_\R$ and $i=1, 2$, let $\h_{\L_i}(\xi) = \sum_{c \in \R} m_{i,\xi, c} c$ be the multi-values of $\h_{\L_1}$, $\h_{\L_2}$ at $\xi$ regarded as elements of the group ring $\Z[\R]$. Then $\h_\L(\xi) = \sum_{c \in \R} (m_{1,\xi,c}+m_{2,\xi,c})c$.   
\end{proposition}

\subsection{Characteristic sequence of polytopes associated to an invariant subspace}
\label{subsec-char-polytopes}
Next we introduce a sequence of polytopes $\Delta_i$, $i=1, \ldots, r$ associated to an invariant subspace $\L = \sum_{\alpha \in \A} E_\alpha \otimes x^\alpha$ of rank $r$ vector-valued Laurent polynomials. The virtual polytopes $\Delta_i-\Delta_{i-1}$ should be thought of as an analogue of the Newton polytopes of a (usual) system of Laurent polynomials. 

Let $1 \leq i \leq r$. We say that an $i$-tuple $(\alpha_1, \ldots, \alpha_i) \in \A^i$ is \emph{admissible} for $\L$ if we can pick $e_{j} \in E_{\alpha_j}$, $j=1, \ldots, i$, such that $\{e_{1}, \ldots, e_{i}\}$ is linearly independent. 

The following lemma 
gives a characterization of admissibility in finite terms.  
\begin{lemma}  \label{lem-equiv-admissible}
For each $\alpha \in \A$ let us fix an (arbitrary) spanning set $B_\alpha$ for $E_\alpha$. For $1 \leq i \leq r$, the following are equivalent:
\begin{itemize}
\item[(a)] $(\alpha_1, \ldots, \alpha_i) \in \A^i$ is admissible.
\item[(b)] There exists $b_j \in B_{\alpha_j}$, $j=1, \ldots, i$, such that $\{b_1, \ldots, b_i\}$ is linearly independent.
\end{itemize}
\end{lemma}
\begin{proof}
(b) $\Rightarrow$ (a) is obvious from the definition. We show (a) $\Rightarrow$ (b). Suppose the $e_j \in E_{\alpha_j}$ are such that $\{e_1, \ldots, e_i\}$ is independent. Since $B_{\alpha_1}$ is a spanning set for $E_{\alpha_1}$, one can find $b_1 \in B_{\alpha_1}$ so that $\{b_1, e_2, \ldots, e_i\}$ is independent. Similarly, one can find $b_2 \in B_{\alpha_2}$ such that $\{b_1, b_2, e_3, \ldots, e_i\}$ is independent. Continuing we arrive at an independent set $\{b_1, \ldots, b_i\}$ where the $b_j \in B_{\alpha_j}$.   
\end{proof}

\begin{definition}[Characteristic sequence of polytopes associated to an invariant subspace]   \label{def-char-seq}
For $1 \leq i \leq r$ define the polytope $\Delta_i$ by:
\begin{equation}   \label{equ-Delta-k}
\Delta_i = \conv\{\sum_{j=1}^i \alpha_j \mid (\alpha_1, \ldots, \alpha_i) \in \A^i \text{ is admissible for } \L \}. 
\end{equation}
In particular, $\Delta_1$ is the convex hull of $\A$. We also put $\Delta_0 = \{0\}$. We call the sequence of polytopes $\Delta_1, \ldots, \Delta_r$, the \emph{characteristic sequence of polytopes} associated to $\L$.
\end{definition}

\begin{remark}  \label{rem-char-polytopes-nonempty}
We point out that all the $\Delta_i$ are nonempty.
Because, if for some $1 \leq i \leq r$ we have $\Delta_i = \emptyset$ then $\sum_{\alpha \in \A} E_\alpha$ should have dimension at most $i-1$. This contradicts the assumption that $E = \sum_{\alpha \in \A} E_\alpha$.     
\end{remark}

\begin{remark} \label{rem-char-seq-discrete-data}
Rado's theorem implies that $(\alpha_1, \ldots, \alpha_i)$ is admissible if and only if the following holds: for any $J \subset \{1, \ldots,, i\}$ we have:
$$\dim(\sum_{j \in J} E_{\alpha_j}) \geq |J|,$$
(see also \cite[Section 2.3, Theorem 4]{Askold-irr-comp}). Thus, the $i$-th characteristic polytope $\Delta_i$ only depends on the finite collection of positive integers $\dim(\sum_{\alpha \in \A'} E_\alpha)$, for all nonempty $\A' \subset \A$ with $|\A'| \leq i$. In other words, the characteristic sequence is determined by the discrete data of the finite set of lattice points $\A$ as well as the polymatroid data of the subspace arrangement $\{E_\alpha\}_{\alpha \in \A}$. (See also Section \ref{subsec-AF-polymatroid}).
\end{remark}
\begin{remark}
    The set of $T$-weights of the $i$-th wedge power $\bigwedge^i \L$ coincides with:
$$\A_i = \{\alpha_1+\cdots+\alpha_i \mid (\alpha_1, \ldots, \alpha_i) \text{ is admissible for } \L\}.$$ In particular
$$\Delta_1(\bigwedge^i\L)=\Delta_i(\L).$$
Indeed from the definition of $\bigwedge^i \L$ we see that $\alpha$ is a $T$-weight of $\bigwedge^i \L$ if and only if there exists $(\alpha_1, \ldots, \alpha_i) \in \A^i$ such that $E_{\alpha_1} \wedge \cdots \wedge E_{\alpha_i} \neq \{0\}$. But this is the case if and only if there exists 
$e_j \in E_{\alpha_j}$ such that $\{e_1, \ldots, e_i\}$ is linearly independent, in other words, $(\alpha_1, \ldots, \alpha_i)$ is admissible.
\end{remark}

We now show that the characteristic sequence of polytopes of $\L$ encodes the canonical representation $h_1\le \cdots \le h_r$ of $\h_\L$. More precisely,

\begin{theorem}  \label{th-h-i-supp-function}
For any $i=1, \ldots, r$, the support function of the virtual polytope $\Delta_i - \Delta_{i-1}$ coincides with the piecewise linear function $h_i$.
\end{theorem}

We note that the characteristic sequence of polytopes of $\L$ and the multi-valued support function $h_\L$ are constructed in seemingly different ways and it is not immediately apparent that they are related. 
 
\begin{proof}
We need to show that
$$h_{\Delta_i} = h_1 + \cdots + h_i.$$
By definition of $\Delta_i$ as a convex hull, for any $\xi \in N_\R$ we have 
$$h_{\Delta_i}(\xi) = \min\{ \langle \xi, \alpha_1 + \cdots + \alpha_i \rangle \mid E_{\alpha_1} \wedge \cdots \wedge E_{\alpha_i} \neq \{0\}  \}.$$ 
For each $\alpha \in \A_\alpha$, let $B_\alpha$ be a spanning set for the subspace $E_\alpha$. By Lemma \ref{lem-equiv-admissible}, we have:
\begin{align*}
h_{\Delta_i}(\xi) = \min\{ \langle \xi, \alpha_1 + \cdots + \alpha_i \rangle \mid \exists b_j \in B_{\alpha_j}, j=1, \ldots, i, \text{ such that }  \{b_1, \ldots, b_i\} \text{ is linearly independent}\}.  
\end{align*}
Suppose the above mininmum is attained at $(\beta_1, \ldots, \beta_i) \in \A^i$. We rewrite the above  equation as follows. Consider the disjoint union $G = \sqcup_{\alpha \in \A} B_\alpha$. In other words, for a vector $b$, we have a copy of $b$ in $G$ for each $\alpha$ such that $b \in B_\alpha$. For $b \in G$ with $b \in B_\alpha$ we put $\alpha_b = \alpha$. Then 
\begin{align*} 
h_{\Delta_i}(\xi) = \min\{ \langle \xi, \alpha_{b_1} + \cdots + \alpha_{b_i} \rangle \mid \{b_1, \ldots, b_i\} \subset G \text{ is linearly independent}\}.  
\end{align*}
Because the linearly independent sets $\{b_1,\ldots,b_i\}$ are determined by a matroid (the vector matroid associated to $G$), the sequence $\beta_1,\ldots,\beta_i$ which minimizes $\langle \xi, \beta_1+ \cdots+ \beta_i \rangle$ can be computed by the well known matroid greedy algorithm:
assume for some $j$ we have already selected linearly independent vectors $b_1,\ldots,b_{j-1}$ associated to characters $\beta_1,\ldots,\beta_{j-1}$. Then we take $b_j$, with associated character $\beta_j$, to be the vector linearly independent from $b_1,\ldots,b_{j-1}$ which minimizes $\langle \xi,\beta_j\rangle$. But this greedy algorithm is easily seen to also compute the sequence  $h_1(\xi),h_1(\xi)+h_2(\xi),\ldots,h_1(\xi)+\cdots+h_i(\xi)$, and hence the result follows.
\end{proof}


\begin{remark}   \label{rem-convex-chain-L}
It follows from Theorem \ref{th-h-i-supp-function} that the convex chain $C_\L: M_\R \to \R$ associated to the multi-valued support function $h_\L: N_\R \to \R$ is the sum of the convex chains of the virtual polytopes $\Delta_i - \Delta_{i-1}$ (see \cite{KhPu} for the notion of convex chain of a virtual polytope).      
\end{remark}

Next we can construct a fan $\Sigma_\L$ such that the support functions $h_i$ are piecewise linear with respect to $\Sigma_\L$.
\begin{definition}[Fan associated to an invariant subspace] 
Let $\Delta_\L = \Delta_1 + \cdots + \Delta_r$ be the Minkowski sum of all the polytopes in the characteristic sequence. 
We let $\Sigma_\L$ be the normal fan of the polytope $\Delta_\L$ (we recall that in the definition of the support function and normal fan we use the minimum convention).  
\end{definition}

Note that without loss of generality we can assume that $\Delta_1$ and hence $\Delta_\L$ is full dimensional: if $\Delta_1$ is not full dimensional, all the monomials corresponding to lattice points in $\Delta_1$ vanish on a non-trivial subtorus $T' \subset T$. Thus, the variety $Y(\f)$ is a direct product of 
$T'$ and a subvariety in the factor torus $T/T'$ defined by $\f$. So it is enough to study this subvariety in the factor torus $T/T'$. 


The following is an immediate corollary of Theorem \ref{th-h-i-supp-function}:
\begin{proposition}
All the support functions $h_i$ are piecewise linear with respect to the fan $\Sigma_\L$.    
\end{proposition}

\begin{proposition}  \label{prop-filt-ind-xi}
Let $\sigma \in \Sigma_\L$. As we vary $\xi \in \sigma^\circ$, the relative interior of $\sigma$, the flag $F^\xi_\bullet$ (Definition \ref{def-parital-flag-L}) does not change.  \end{proposition}
\begin{proof}
By Theorem \ref{th-h-i-supp-function}, if we vary $\xi$ in the relative interior of a cone $\sigma \in \Sigma_\L$, the numbers $h_1(\xi) \leq \ldots \leq h_r(\xi)$ in \eqref{equ-c-i} vary linearly in $\xi$ and hence the filtration $F^\xi_\bullet$ does not change.
\end{proof}

We will need the following later. It describes the charactersitic sequence of polytopes of an invariant subspace under a generic linear projection. Let $\L = \bigoplus_{\alpha \in \A} E_\alpha \otimes x^\alpha \subset E \otimes \k[T]$ be an invariant subspace. Let $\pi:E \to E'$ be a surjective linear map. Let us put $L' = \pi(\L) = \bigoplus_{\alpha \in \A} \pi(E_\alpha) \otimes x^\alpha$. We say that a subspace $W \subset E$ is \emph{generic} with respect to $\{E_\alpha\}_{\alpha \in \A}$ if $W$ is transverse to all the subspaces $\sum_{\alpha \in \A'} E_\alpha$, for all $\A' \subset \A$. 
\begin{proposition}  \label{prop-char-seq-proj}
Let $(\Delta_1, \ldots, \Delta_r)$ be the characteristic sequence of polytopes of an invariant subspace $\L$. Let $\pi: E \to E'$ be a surjective linear map such that $\ker(\pi)$ is generic with respect to the subspace arrangement $\{E_\alpha\}_{\alpha \in \A}$.
Then the characteristic sequence of the invariant subspace $\L' = \pi(\L)$ is equal to $(\Delta_1, \ldots, \Delta_{r'})$ where $r' = \dim(E')$.    
\end{proposition}
\begin{proof}
Let $1 \leq i \leq r'$ and let $(\alpha_1, \ldots, \alpha_i)$ be admissible for $\L$ but not admissible for $\L' = \pi(\L)$. Then by Remark \ref{rem-char-seq-discrete-data} (Rado's theorem) there exists $J \subset \{1, \ldots, i\}$ such that $\dim(\pi(\sum_{j \in J} E_{\alpha_j})) < |J|$. Let $V = \sum_{j \in J} E_{\alpha_j}$ and $W = \ker(\pi)$. If $\dim(V) + \dim(W) < r$ then transversality of $V$ and $W$ implies that $V \cap W = \{0\}$ and hence $\pi: V \to E'$ is one-to-one. This contradicts the assumption that $(\alpha_1, \ldots, \alpha_i)$ is admissible for $\L$. Thus, we must have $\dim(V) + \dim(W) \geq r$, that is, $\dim(V) \geq r'$. But in this case, transversality of $V$ and $W$ implies that $\dim(\pi(V)) = r'$ and hence $\pi: V \to E'$ is onto. This contradicts the assumption that $\dim(\pi(V)) < |J| \leq i \leq r'$. 

\end{proof}

\section{Non-degenerate systems}    \label{sec-non-degen-systems}
In this section we discuss the non-degeneracy of a vector-valued system $\f(x)=0$, for $\f \in \L$. We introduce conditions on $\f \in \L$ that guarantee that the closure of the subvariety $Y(\f)=\{x \in T \mid \f(x) =0\}$, in a toric compactification that is sufficiently refined with respect to $\L$, is smooth and transverse to all orbits. 

\subsection{Truncated systems}  \label{subsec-trunc-infinity}
We start with a linear algebra technique. It concerns associating a direct sum decomposition $E = V_1 \oplus \cdots \oplus V_k$ to a flag of subspaces $F_\bullet=(\{0\}=F_0 \subsetneqq F_1 \subsetneqq \cdots \subsetneqq F_k=E)$. We will use this to break a vector-valued system $\f=0$, $\f \in \L \subset E \otimes \k[T]$, into smaller vector-valued systems $\f_i = 0$, $i=1, \ldots, k$, where $\f_i \in V_i \otimes \k[T]$, is obtained by projecting the coefficients of $\f$ into $V_i$.

As before, let $\L = \bigoplus_{\alpha \in \A} E_\alpha x^\alpha$. For the rest of the section, we fix a complete flag:
$$W_\bullet = (\{0\}=W_0 \subsetneqq W_1 \subsetneqq \cdots \subsetneqq W_r = E),$$
such that any subspace $W_i$ is generic with respect to the subspace arrangement $\{E_\alpha\}_{\alpha \in \A}$, in the sense that $W_i$ is transverse to all the subspaces $\sum_{\alpha \in \A'} E_\alpha$, for any $\A' \subset \A$. 
Consider any partial flag:
$$F_\bullet = (\{0\}=F_0 \subsetneqq F_1 \subsetneqq \cdots \subsetneqq F_k = E),$$
such that each subspace $F_i$ is a sum of some subspaces $E_\alpha$. We can use the flag $W_\bullet$ to assign (in a unique way) a direct sum decomposition $E = V_1 \oplus \cdots \oplus V_k$ 
such that $F_i = V_1 \oplus \cdots \oplus V_i$, for any $i=1, \ldots, k$. We construct the $V_i$ as follows: let $V_1 = F_1$. Then by the transversality of the complete flag $W_\bullet$, there exists $W_{i_1}$ such that $W_{i_1} \cap F_2$ is a complement of $F_1$ in $F_2$. Let $V_2 = W_{i_1} \cap F_2$. Continuing in this manner we construct subspaces $V_1, \ldots, V_k$ such that $F_i = V_1 \oplus \cdots \oplus V_i$, for all $i=1, \ldots, k$, as desired. 

Let $\pi_i: E \to V_i$ be projection on the $V_i$ component.
Let $\f = \sum_{\alpha} e_\alpha x^\alpha \in \L$ be a vector-valued Laurent polynomial. We write:
$$\f_i = \sum_\alpha \pi_i(\e_\alpha) x^\alpha \in V_i \otimes \k[T].$$
We thus have $\f = \f_1 \oplus \cdots \oplus \f_k$.

Let $\xi \in N_\R$. We are interested in the behavior of the subvariety $Y(\f) \subset T$ in the direction of $\xi$. For this we introduce the notion of $\xi$-truncated system. This is exactly as in the classic Newton polyhedra theory from \cite{Askold-toroidal}.

First let us recall the notion of truncation of a (usual scalar-valued) Laurent polynomial along $\xi \in N_\R$ (see \cite{Askold-toroidal}). Let $g = \sum_{\beta \in \mathcal{B}} c_\beta x^\beta \in \k[T]$ with Newton polytope $P = \conv\{\beta \mid c_\beta \neq 0\}$. We define its \emph{$\xi$-truncation} $g^\xi$ to be the sum of all terms at which the minimum of pairing with $\xi$ is attained. That is:
$$g^\xi = \sum_{\beta \in \mathcal{B} \cap P^\xi} c_\beta x^\beta,$$ where $P^\xi$ is the face of $P$ on which the minimum of $\langle \xi, \cdot \rangle$ is attained.
The above definition extends to vector-valued Laurent polynomials. 

\begin{definition}[$\xi$-truncated system]
With notation as above, for $i=1, \ldots, k$, we let $\f_i^\xi$ to be the vector-valued Laurent polynomial $\f_i \in V_i \otimes \k[T]$ defined by: 
$$\f^\xi_i = \sum_{\langle \xi, \alpha\rangle = c_i} \pi_i(e_\alpha) x^\alpha,$$
where $c_i=c_i^\xi$ is the $i$-th critical number of $(\L, \xi)$ (see Section \ref{subsec-multi-val-supp}).
We put $\f^\xi = \f^\xi_1 \oplus \cdots \oplus \f^\xi_k$ and call it the \emph{$\xi$-truncation of $\f \in \L$}. (Here we assume that $\f_i^\xi \neq 0$ which is the case for a generic choice of $\f \in \L$.)
\end{definition}

\begin{proposition}  \label{prop-truncations-finite}
Fix $\f \in \L$. Let $\sigma \in \Sigma_\L$ and $\xi \in \sigma^\circ$, the relative interior of $\sigma$. Then the $\xi$-truncations $\f^\xi_i$, $i=1, \ldots, k$, only depend on $\sigma$. Thus, as we vary $\xi \in N_\R$, we only get a finite number of truncated systems. For $\xi \in \sigma^\circ$, we also denote the truncated system $\f^\xi_1 \oplus \cdots \oplus \f^\xi_k$ by $\f^\sigma = \f^\sigma_1 \oplus \cdots \oplus \f^\sigma_k$.  
\end{proposition}
\begin{proof}
This is because, as $\xi$ varies in $\sigma^\circ$, the flag $F^\xi_\bullet$ does not change (Proposition \ref{prop-filt-ind-xi}) and moreover, the critical numbers $c_1 < \cdots < c_k$ vary linearly.   
\end{proof}



\begin{remark}  \label{rem-trunc-ind-decomp}
In fact, the $\xi$-truncation of $\f \in \L$ can be defined in such a way that only depends on the flag $F_\bullet = F^\xi_\bullet$ and is independent of the choice of the decomposition $E=V_1 \oplus \cdots \oplus V_k$.    

Recall that $c_1 = \min\{ \langle \xi, \alpha \rangle \mid \alpha \in \A\}$.
Then the first $\xi$-truncation $\f^\xi_1$ is the vector-valued Laurent polynomial $\f^\xi_1 \in F_1 \otimes \k[T]$ given by:
$$\f^\xi_1 = \sum_{\langle \xi, \alpha \rangle = c_1} e_\alpha x^\alpha.$$
Next define $\t{\f}_1 \in (E/F_1) \otimes \k[T]$ by:
$$\t{\f}_1 = \sum_{\langle \xi, \alpha \rangle > c_1} [e_\alpha] x^\alpha,$$ where $[e_\alpha]$ denotes the image of $e_\alpha$ in $E/F_1$. We repeat the above with $E/F_1$ in place of $E$ and $\t{\f}_1$ in place of $\f$ to obtain $\f^\xi_2 \in (F_2/F_1) \otimes \k[T]$ and $\t{\f}_2 \in E/F_2 \otimes \k[T]$. Continuing we get the sequence of $\xi$-truncations $\f^\xi_i \in (F_i/F_{i-1}) \otimes \k[T]$, $i=1, \ldots, k$.
\end{remark}



Let $\sigma \subset N_\R$ be a cone that is contained in a cone in the fan $\Sigma_\L$ and let $\xi$ lie in the relative interior of $\sigma$. 
As above, let us write $\f = \f_1 \oplus \cdots \oplus \f_k$. Let $P_i$ denote the Newton polytope of $\f_i$, that is, the convex hull of exponents of the monomials in $\f_i$ appearing with nonzero (vector) coefficients. Since we assumed that $\f_i^\xi \neq 0$ there is a vertex $\beta_i$ in $P_i$ such that $\langle \xi, \beta_i\rangle = c_i$. We have the following key lemma about the support function of $P_i$. 
\begin{lemma}  \label{lem-P-i-sigma}
Let $h_{P_i}$ denote the support function of the Newton polytope $P_i$.
\begin{itemize}
\item[(a)] For any $\xi \in \sigma^\circ$, $h_{P_i}(\xi)$ is a linear function of $\xi$ and is given by the $i$-th critical number $c_i=c_i^\xi$.
\item[(b)] Let $\beta_i$ be a vertex of $P_i$ such that $\langle \xi, \beta_i \rangle = c_i$. Then the shifted polytope $(-\beta_i)+P_i$ lies in the dual cone $\sigma^\vee$. Moreover, the minimal face $(-\beta_i)+P_i^\xi$, defined by $\langle \xi, \cdot \rangle = 0$, lies on $\sigma^\perp$.
\item[(c)] When $\sigma$ is full dimensional, $E_{\beta_1} + \cdots + E_{\beta_k} = E$.
\end{itemize}
\end{lemma}
\begin{proof}
Since the support functions $h_1 \leq \cdots \leq h_r$ are linear on the cone $\sigma$, the critical values $c_i$ vary linearly as $\xi$ varies in $\sigma^\circ$. Thus, we have: 
\begin{equation} \label{equ-xi-c-i-2}
\langle \xi, \beta_i \rangle = c_i,~~\forall \xi \in \sigma^\circ.  \end{equation}
But by definition of support function of a polytope, we have $h_{P_i}(\xi) = \langle \xi, \beta_i \rangle$. This shows (a).
For any $\xi \in \sigma^\circ$, the flag $F^\xi_\bullet$ is the same and hence the direct sum decomposition $E = V_1 \oplus \cdots \oplus V_k$ is also the same. Thus, by construction, for any $i=1, \ldots, k$ and for any $\alpha$ in the support of $\f_i$ we have:
\begin{equation}  \label{equ-xi-c-i}
\langle \xi, \alpha \rangle \geq c_i,~~  \forall \xi \in \sigma^\circ,
\end{equation}
The equations \eqref{equ-xi-c-i-2} and \eqref{equ-xi-c-i} imply that the shifted polytope $(-\beta_i)+P_i$ lies in the dual cone $\sigma^\vee$ and the face $(-\beta_i)+P_i^\xi$ lies on $\sigma^\perp$ which proves (b). To show (c) we note that since $\sigma$ is full dimenaional, for each $i$, the face $P_i^\xi$, where $\langle \xi, \cdot \rangle = c_i$, consists of the single vertex $\beta_i$. This implies that $E_{\beta_i}$ projects surjectively onto $V_i$ and hence $E_{\beta_1} + \cdots + E_{\beta_k} = V_1 \oplus \cdots \oplus V_k = E$, as required.     
\end{proof}

The above lemma implies the following important statement about the behavior of a vector-valued polynomial on an affine toric chart. It plays a main role in the proof of Theorem \ref{th-non-degen-infinity}, as well as, Section \ref{sec-tvb} in the construction of vector bundle $\E_{\L, \Sigma}$ associated to $\L$. 

\begin{theorem}  \label{th-f-U-sigma}
With notation as above, we have the following: 
\begin{itemize}
\item[(a)]
The map:
$$\bigoplus_{i=1}^k x^{-\beta_i} \f_i: T \to \bigoplus_{i=1}^k V_i = E,$$
extends to a regular map from $U_\sigma$ to $E$, where $U_\sigma$ is the affine toric variety corresponding to $\sigma$. 
\item[(b)] When $\sigma$ is a full dimensional cone, we can vary $\f$ in such a way that the constant term of $\oplus_i x^{-\beta_i} \f_i$ be any vector in $E =V_1 \oplus \cdots \oplus V_k$, without changing the coefficients of the rest of the monomials in it. Moreover, for generic $\f \in \L$, the vector-valued Laurent polynomial $\oplus_i x^{-\beta_i} \f_i$ has a nonzero constant term.   
\end{itemize}
\end{theorem}
\begin{proof}
(a) By Lemma \ref{lem-P-i-sigma}(a), for any $\alpha$ in the support of $\f_i$, the monomial $x^{-\beta_i+\alpha}$ is a regular function on $U_\sigma$. The claim immediately follows from this.
(b) The claim follows from Lemma \ref{lem-P-i-sigma}(c).
\end{proof}

Finally, we describe the characteristic sequence of polytopes for a truncated system. For $1 \leq i \leq k$, let $\L^\xi_i \subset V_i \otimes \k[T]$ be the invariant subspace spanned by all the truncations $\f^\xi_i$, $\f \in \L$. In other words:
$$\L^\xi_i = \bigoplus_{\langle \xi, \alpha \rangle = c_i} \pi_i(E_\alpha) \otimes x^\alpha.$$
The characteristic sequence of $\L$ in fact determines the characteristic sequence of $\L^\xi_i$.  

\begin{theorem}[Characteristic sequence of a truncated system]  \label{th-char-seq-trunc} 
With notation as before, let $\Delta^i_1, \ldots, \Delta^i_{d_i-d_{i-1}}$ be the characteristic sequence of $\L^\xi_i$. Then we have the following equality, for $1\leq i \leq k$ and $d_{i-1} < j \leq d_i$:
$$\Delta^\xi_{j} = \Delta^i_{j-d_{i-1}} +  \sum_{\ell=1}^{i-1} \Delta^\xi_{d_\ell}.$$
Thus, $\Delta^i_{j-d_{i-1}}$ can be expressed as the Minkowski difference $\Delta^\xi_{j} - \sum_{\ell=1}^{i-1} \Delta^\xi_{d_\ell}$. In particular, 
$$\Delta^i_{j-d_{i-1}} - \Delta^i_{j-1-d_{i-1}} = \Delta^\xi_{j} - \Delta^\xi_{j-1}.$$
Recall that for a polytope $P$, the polytope $P^\xi$ is the face of $P$ on which the minimum of $\langle \xi, \cdot \rangle$ is attained.
\end{theorem}
\begin{proof}
Firstly, for $i=1$ and $j=1, \ldots, d_1$, we have:
$$\Delta^1_j = \conv\{\alpha_1+\cdots+\alpha_j \mid \langle \xi, \alpha_1 \rangle = \cdots = \langle \xi, \alpha_j \rangle = c_1,  \textup{ and } (\alpha_1, \ldots, \alpha_j) \textup{ is admissible for } F_1=V_1 = \sum_{\langle \xi, \alpha \rangle = c_1} E_\alpha \}.$$
But since $\langle \xi, \alpha_1 \rangle = \cdots = \langle \xi, \alpha_j \rangle = c_1$, the admissibility of $(\alpha_1, \ldots, \alpha_j)$ with respect to $V_1 = \sum_{\langle \xi, \alpha \rangle = c_1} E_\alpha$ is equivalent to the admissibility with respect to the whole $E = \sum_\alpha E_\alpha$. This shows $\Delta^1_j = \Delta_j^\xi$, for $j=1, \ldots, d_1$.
Next, for $i=2$ and $j=d_1+1, \ldots, d_2$, we have:
\begin{multline*}
\Delta^\xi_{d_1} + \Delta^2_{j-d_1} = \conv\{\alpha_1+\cdots+\alpha_j \mid \langle \xi, \alpha_1 \rangle = \cdots = \langle \xi, \alpha_{d_1} \rangle = c_1,~~ \langle \xi, \alpha_{d_1+1} \rangle = \cdots = \langle \xi, \alpha_j \rangle = c_2 \\ \textup{ and } (\alpha_1, \ldots, \alpha_j) \textup{ is admissible for } F_2=V_1 \oplus V_2 = \sum_{\langle \xi, \alpha \rangle = c_1}  E_\alpha + \sum_{\langle \xi, \beta \rangle = c_2} E_\beta\}.    
\end{multline*}
Again, since $\langle \xi, \alpha_1 \rangle = \cdots = \langle \xi, \alpha_{d_1} \rangle = c_1,~~ \langle \xi, \alpha_{d_1+1} \rangle = \cdots = \langle \xi, \alpha_{j} \rangle = c_2$, the admissibility of $(\alpha_1, \ldots, \alpha_j)$ with respect to $V_1 \oplus V_2$ is equivalent to the admissibility with respect to the whole $E = \sum_\alpha E_\alpha$. This shows $\Delta^\xi_{d_1}+\Delta^2_{j-d_1} = \Delta_j^\xi$, for $j=d_1+1, \ldots, d_2$. Continuing in this fashion proves the claim.
\end{proof}

\subsection{$\L$-non-degeneracy}  \label{subsec-L-non-degen}

Let $\f \in E \otimes \k[T]$. We say that a vector equation $\f(x) = 0$ is \emph{non-degenerate} if for any solution $p \in Y(\f)$, the differential $d\f_p: T_pT \to E$ is surjective. 

\begin{definition}[$\L$-non-degenerate system]
\label{def-L-non-degen}
For $\xi \in N_\R$, we say that the vector equation $\f(x) = 0$ is \emph{$\xi$-non-degenerate} if the truncated system $\f^\xi = \f^\xi_1 \oplus \cdots \oplus \f^\xi_k = 0$ is non-degenerate.
For a cone $\sigma$ that lies in a cone in $\Sigma_\L$, we say that the vector equation $\f(x) = 0$ is \emph{$\sigma$-non-degenerate} if it is $\xi$-non-degenerate for some (and hence any) $\xi$ in $\sigma^\circ$, the relative interior of $\sigma$ (see Proposition \ref{prop-truncations-finite}). Finally, we say that $\f \in \L$ is \emph{$\L$-non-degenerate} if it is $\xi$-non-degenerate for all $\xi \in N_\R$.
\end{definition}



The next theorem is the main result of this section. It states that $\L$-non-degeneracy is a generic condition and implies smoothness and transversality to all the orbits in a sufficiently refined toric compactification. 
\begin{theorem}  \label{th-non-degen-infinity}
With notation as above, we have the following:
\begin{itemize}
\item[(a)] Let $\sigma$ be a smooth cone contained in a cone in the fan $\Sigma_\L$. Let $\f \in \L$ be $\xi$-non-degenerate for all $\xi \in \sigma$ (that is, it is non-degenerate with respect to $\sigma$ and all its faces). Then the closure of the subvariety $Y(\f)$ in the affine toric variety $U_\sigma$ is smooth and transverse to all the orbits.
\item[(b)] Let $\Sigma$ be a smooth fan that refines $\Sigma_\L$. Let $\f \in \L$ be $\L$-non-degenerate (in other words, it is non-degenerate with respect to all $\sigma \in \Sigma$). Then the closure of the subvariety $Y(\f)$ in the complete toric variety $X_{\Sigma}$ is smooth and transverse to all the orbits.

\item[(c)] $\L$-non-degeneracy is a generic condition, that is, there is a non-empty Zariski open subset in $\L$ such that any $\f$ in this open subset is $\L$-non-degenerate.
\end{itemize}
\end{theorem}
\begin{proof}
(a) Without loss of generality assume $\sigma$ is full dimensional. By assumption, $\sigma$ is a smooth cone and hence, after a monomial change of coordinates in $T$, we can assume that $\sigma$ is the positive orthant $\R_{\geq 0}^n$ and $U_\sigma$ is the affine space $\k^n$. 
We follow notation in Section \ref{subsec-trunc-infinity}. 
By non-degeneracy of $\f_i^\xi$ we know that $\f_i^\xi \neq 0$ and hence $x^{-\beta_i} \f_i(x)$ has a nonzero term $e_i x^{\beta_i}$, $e_i \in V_i$. Let $\g_i = x^{-\beta_i}\f_i - e_i$. Then, on the torus $T$, the system $\f_i(x) = 0$, $i=1, \ldots, k$, is equivalent to the system $\g_i(x) = -e_i$, $i=1, \ldots, k$. Moreover, $\g = \g_1 \oplus \cdots \oplus \g_k$ is a regular map on $U_\sigma \cong \k^n$. The condition that $\f$ is $\xi$-non-degenerate for all the $\xi \in \sigma$ is the same as the condition that $e=(e_1, \ldots, e_k)$ is a regular value for the mappings $\g_{|O_\tau}$, for all the orbits $O_\tau$ in $U_\sigma$. The claim now follows from usual arguments in differential geometry. (b) It follows from (a). (c) To prove $\L$-non-degeneracy is a generic condition, it suffices to show that for given $\xi \in N_\R$, the $\xi$-non-degeneracy is a generic condition (Proposition \ref{prop-truncations-finite}). With notation as in the proof of (a), let $\tilde{\L}$ be the subspace consisting of all $\g=\g_1 \oplus \cdots \oplus \g_k$, $\forall \f \in \L$. Consider the map $\pi: \L \to \tilde{\L}$, $\f \mapsto \g$ and let us fix some $\g \in \tilde{\L}$. We note that, by Theorem \ref{th-f-U-sigma}(b), if $\f \in \pi^{-1}(\g)$ then, for any choice of $e_i \in V_i$, $\f+\sum_{i=1}^k e_i x^{\beta_i}$ also belongs to $\pi^{-1}(\g)$. Now, by the Bertini-Sard theorem, one knows that there is a non-empty Zariski open set $U_{\g} \subset V_1 \oplus \cdots \oplus V_k$ such that any $c=(c_1, \ldots, c_k) \in U_\g$ is a regular value for all the $\g_{|O_\tau}$. Let $U_\xi$ denote the subset of $\L$ consisting of $\xi$-non-degenerate $\f$. By the above argument, intersection of $U_\xi$ with each fiber $\pi^{-1}(\g)$ contains a nonempty Zariski open subset in that fiber. It then follows that $U_\xi$ itself should contain a nonempty Zariski open subset of $\L$ itself (because every semi-algebraic set either contains a Zariski open or is contained in a Zariski closed). This finishes the proof. 
\end{proof}

\begin{corollary}  \label{cor-system-at-infinity}
With notation as above, let $\sigma$ be a smooth cone contained in a cone in the fan $\Sigma_\L$. Let $\f \in \L$ be $\xi$-non-degenerate for all $\xi \in \sigma$. then  
the closure of $Y(\f)$ in the affine toric variety $U_\sigma$, is defined by the system of equations:
$$x^{-\beta_i} \f_i(x) = 0, ~~i=1,\ldots,k.$$
Moreover, the $\sigma$-truncated system $\f^\sigma_i(x) = 0$, $i=1, \ldots, k$, gives equations for the intersection of the closure of $Y(\f)$ in the affine toric variety $U_{\sigma}$ with the orbit $O_\sigma \cong T / T_\sigma$, where $T_\sigma$ is the stabilizer of the orbit $O_\sigma$.
\end{corollary}
\begin{proof}
By definition, the subvariety $Y(\f) \subset T$ is defined by the equation $\f(x) = 0$. But $\f(x)=0$ is equivalent to $\f_i(x) = 0$, $i=1, \ldots, k$, which in turn is equivalent to $x^{-\beta_i} \f_i(x) = 0$. 
Since the closure of $Y(\f)$ is transverse to all the orbits, it follows that the closure of $Y(\f)$ is given by the system of equations $x^{-\beta_i}\f_i=0$, $i=1, \ldots, k$, as claimed (Theorem \ref{th-f-U-sigma}).
The last claim follows from the above. 
\end{proof}


Finally, we have the following key statement which states that the subvarieties defined by $\L$-non-degenerate systems are all diffeomorphic (over $\R$). Thus, they have the same topological invariants. 
\begin{theorem}  \label{th-non-degenerate-diffeo}
Let $\f$, $\f' \in \L$ be $\L$-non-degenerate. Then the following is true:
\begin{itemize}
\item[(a)] The subvarieties $Y(\f)$ and $Y(\f')$ are diffeomorphic.
\item[(b)] Let $\Sigma$ be a smooth complete fan refining $\Sigma_\L$. Then the closures of $Y(\f)$ and $Y(\f')$ in $X_{\Sigma}$ are diffeomorphic. 
\item[(c)] Moreover, if $\Sigma$ is not assumed to be smooth, then the closures of $Y(\f)$ and $Y(\f')$ are homeomorphic via a homeomorphism which respects the stratification by orbits and restricts to diffeomorphisms on intersections with orbits in $X_{\Sigma}$.  
\end{itemize}
\end{theorem}
\begin{proof}[Sketch of proof]
One can show that the $\L$-non-degeneracy condition is open, that is, the locus of $\L$-non-degenerate $\f$ is a non-empty Zariski open subset of $\L$. Thus, the locus of $\L$-degenerate systems has real codimension at least $2$ which implies that the $\L$-non-degenerate locus is path connected. Take a path connecting $\f$ and $\f'$. The claims now follows from usual arguments from differential topology (applied to the compact manifold $X_{\Sigma}$). In case $X_{\Sigma}$ is singular, one considers its stratification by torus orbits, which are smooth submanifolds. 
\end{proof}

\section{A vector-valued BKK theorem} \label{sec-BKK}
As before, let $\L = \bigoplus_{\alpha \in \A}E_\alpha \otimes x^\alpha \subset E \otimes \k[T]$ be a $T$-invariant subspace. Recall that $(\Delta_1, \ldots, \Delta_r)$ denotes its characteristic sequence of polytopes as defined in Section \ref{subsec-char-polytopes}. Also we let $\Delta = \Delta_1 + \cdots + \Delta_r$ and $\Sigma_\L$ is the normal fan of the polytope $\Delta$.

Before stating the main theorem, we need to set a notation. Let $V \subset \R^n$ be an $r$-dimensional rational subspace. Let $P_1, \ldots, P_r$ be convex polytopes in $\R^n$ that lie in affine hyperplanes parallel to $V$. We define the mixed volume $\MVol_r(P_1, \ldots, P_r)$ to be the mixed volume $\MVol_r(P'_1, \ldots, P'_r)$ where $P'_i \subset V$, $i=1, \ldots, r$, denotes an arbitrary translation of $P_i$ that lies in $V$, and $\MVol_r$ is the mixed volume in the real vector space $V$ normalized with respect to the lattice $V \cap \Z^n$. The definition extends to virtual polytopes by multi-linearity.

\begin{theorem}[Vector-valued BKK theorem] \label{th-vec-BKK}
Suppose $r=n$. If $\f \in \L$ is $\L$-non-degenerate then the subvariety $Y(\f)$ consists of a finite number of points and $|Y(\f)|$ is given by:
$$|Y(\f)| = n!\MVol_n(\Delta_1, \Delta_2-\Delta_1, \cdots, \Delta_n-\Delta_{n-1}).$$
\end{theorem}
\begin{remark} 
If we only assume that $\f$ is $\xi$-non-degenerate for all $0 \neq \xi \in N_\R$ then the number of solutions is finite and counted with multiplicity is given by the same answer.
\end{remark}

Theorem \ref{th-vec-BKK} implies the following more refined statements.
\begin{theorem}   \label{th-BKK-vec2}
Let $\L$ be of rank $r \leq n$. Let $\Sigma$ be a complete fan refining $\Sigma_\L$ and let $\f \in \L$.
\begin{itemize}
\item[(a)] 
Let $\f \in \L$ be such that for all $\sigma \in \Sigma$ with $\dim(\sigma) \geq n-r$, the subvariety $Y(\f)$ is $\sigma$-non-degenerate. 
Let $P_1, \ldots, P_{n-r}$ be lattice polytopes and let $f_i$, $i=1,\ldots, n-r$, be generic Laurent polynomials with Newton polytopes $P_i$.
Then the number of points in $$Y(\f, f_1, \ldots, f_{n-r}) = \{x \in T \mid \f(x) = 0 \textup{ and } f_1(x) = \cdots = f_{n-r}(x) = 0 \},$$
is equal to:
$$n!\MVol_n(\Delta_1, \Delta_2-\Delta_1, \cdots, \Delta_r-\Delta_{r-1}, P_1, \ldots, P_{n-r}).$$

\item[(b)] More generally, we can take several invariant subspaces. Take $1 \leq p \leq n$ and let $E_1, \ldots, E_p$ be $\k$-vector spaces of dimensions $r_1, \ldots, r_p$ respectively with $\sum_{i=1}^p r_i = n$.
Let us take $T$-invariant subspaces $\L_i \subset E_i \otimes \k[T]$, $i=1, \ldots, p$.
For $i=1, \ldots, p$, let $\Delta^i_1, \ldots, \Delta^i_{r_i}$ be the characteristic polytopes associated to $\L_i \subset E_i \otimes \k[T]$. Let $\f_i \in \L_i$ be such that $\f=\f_1 \oplus \cdots \oplus \f_p$ is $(\L_1\oplus \cdots \oplus \L_p)$-non-degenerate. Then the number of solutions $x \in T$ of the system $\f_i(x) = 0$, $i=1, \ldots, p$, for generic $\f_i \in \L_i$,  is equal to:
$$n! \MVol(\Delta^1_1,\Delta^1_2 - \Delta^1_1, \ldots, \Delta^1_{r_1} - \Delta^1_{r_1-1}, \ldots, \Delta^p_{r_p} - \Delta^p_{r_p - 1}).$$
\item[(c)] The image of the subvariety $Y(\f)$ in the ring of conditions of the torus $T$ is given by the product of the classes of the virtual polytopes $\Delta_i - \Delta_{i-1}$:
$$[Y(\f)] = \prod_{i=1}^r [\Delta_i - \Delta_{i-1}].$$
\end{itemize}
\end{theorem}
\begin{proof}
(a) Let $L_i$ be the subspace of (usual) Laurent polynomials with Newton polytopes contained in $P_i$. Let $\tilde{\L} = \L \oplus L_1 \oplus \cdots \oplus L_{n-r}$ and let $\tilde{\Delta}_1, \ldots, \tilde{\Delta}_n$ be its characteristic sequence. By Theorem \ref{th-vec-BKK} we know that the number of solutions of the system is given by $n!$ times the mixed volume of the virtual polytopes $\tilde{\Delta}_i - \tilde{\Delta}_{i-1}$, which in turn is equal to $n!$ times the mixed volume of the multi-valued support function  $\h_{\tilde{\L}}$ (Theorem \ref{th-h-i-supp-function}). But, by Proposition \ref{prop-supp-function-direct-sum}, the multi-valued support function $\h_{\tilde{\L}}$ is obtained by merging $\h_\L$ and the support functions of the $P_i$ which implies $\MVol(\h_{\tilde{\L}}) = \MVol_n(\Delta_1, \Delta_2-\Delta_1, \cdots, \Delta_r-\Delta_{r-1}, P_1, \ldots, P_{n-r})$. This proves the claim. Proof of (b) is the same as (a). Part (c) follows from (a). 
\end{proof}

\begin{theorem}  \label{th-Minkowski-weights}
Let $\Sigma$ be a complete fan refining $\Sigma_\L$. Let $\f \in \L$ be such that for all $\sigma \in \Sigma$ with $\dim(\sigma) \geq n-r$, the subvariety $Y(\f)$ is $\sigma$-non-degenerate. In particular, when $\dim(\sigma) > n-r$, the closure $\overline{Y}_\f$ in $X_{\Sigma}$ does not intersect the orbit $O_\sigma$. Then for any $\sigma \in \Sigma$ with $\dim(\sigma) = n-r$, the number of intersections of $\overline{Y}_\f$ with the orbit $O_\sigma$ is equal to:
$$\MVol_r(\Delta^\xi_1, \Delta^\xi_2 - \Delta^\xi_1, \ldots, \Delta^\xi_r - \Delta^\xi_{r-1}),$$
where $\xi$ is any vector in the relative interior of $\sigma$. 
\end{theorem}
\begin{proof}
Corollary \ref{cor-system-at-infinity} tells us that the intersection of the closure of $Y(\f)$ with the orbit $O_\sigma$ is given by the truncated system. Also, Theorem \ref{th-char-seq-trunc} tells us what the characteristic sequence of each truncation $\f^\xi_i$ is. The claim now follows from Theorem \ref{th-BKK-vec2}(b) applied to the orbit $O_\sigma$.      
\end{proof}

\begin{proof}[Proof of Theorem \ref{th-vec-BKK}] 
The proof follows the same idea as \cite[Section 27.8]{Askold-BZ}.
We prove the claim by induction on $r$. 
Pick a nonzero vector $w_1 \in E$ such that the $1$-dimensional subspace $W_1$ spanned by $w_1$ is generic with respect to the subspace arrangement $\{E_\alpha\}_{\alpha \in \A}$ (that is, $W_1$ is transverse to any sum of the subspaces $E_\alpha$). Let $\pi: E \to E/W_1$ denote the natural projection. Let $\bar{\f} = \pi(\f) \in (E/W_1) \otimes \k[T]$. Then:
$$Y_{\bar{\f}} = \{x \in T \mid \f(x) \in W_1\}.$$
Let $\Sigma$ be a smooth fan refining $\Sigma_\L$. By Theorem \ref{th-non-degen-infinity}, for $\f \in \L$ that is $\sigma$-non-degenerate for all $\sigma \in \Sigma$, the subvariety $C = Y_{\bar{\f}}$ is a curve whose closure $\overline{C} \in X_{\Sigma}$ is smooth and transverse to all the orbits. 

Let $\Delta_1, \ldots, \Delta_n$ be the characteristic sequence of the invariant subspace $\L$.
We recall that, by Proposition \ref{prop-char-seq-proj}, the characteristic sequence of the projected invariant subspace $\pi(\L)$ is equal to $\Delta_1, \ldots, \Delta_{n-1}$.

Now consider the regular function $g: C \to \k$ defined by the equality:
$$\f(x) = g(x) w_1, ~~\forall x \in C.$$
Since $\f$ is assumed to be non-degenerate, we know that $Y(\f) = \{x \in T \mid \f(x) = 0\}$ is finite. From construction, it is clear that $|Y(\f)|$ is equal to the number of zeros of the function $g$ on the curve $C$. 
We recall that the sum of order of zeros/poles of $g$ (regarded as a rational function on the smooth projective curve $\overline{C}$) is equal to $0$. It follows that:
\begin{equation}  \label{equ-sum-poles}
|Y(\f)| = -\sum_{p \in \overline{C} \setminus C} \ord_p(g).    
\end{equation}
Since the curve $\overline{C}$ is transverse to all the orbits, it only intersects codimension $1$ orbits $O_\rho$, $\rho \in \Sigma_\L(1)$. 
\begin{lemma}  \label{lem-ord-g}
For any codimension $1$ orbit $O_\rho$, the rational function $g$ has the same order of zero/pole at all the $p \in \overline{C} \cap O_\rho$ equal to $$\ord_p(g) = h_n(\xi)=c^\xi_k, ~~\forall p \in \overline{C} \cap O_\rho,$$ where $\xi$ is the primitive generator of the ray $\rho$.
\end{lemma}
\begin{proof}
As in Section \ref{subsec-trunc-infinity}, pick a generic complete flag $W_\bullet = (W_1 \subsetneqq \cdots \subsetneqq W_r = E)$, where the first subspace is $W_1$. Consider the direct sum decomposition $E = V_1 \oplus \cdots \oplus V_k$ compatible with the flag $F_\bullet^\xi$ associated to $\xi$ and corresponding to this choice of $W_\bullet$. By construction, $V_k$ contains $W_1$. If $V_k \neq W_1$, then further refine the decomposition by writing $V_k = V'_k \oplus W_1$, for some generic subspace $V'_k$. Thus, without loss of generality, we may assume that $V_k = W_1$. Now, following the notation in Section \ref{subsec-trunc-infinity}, we write $\f = \f_1 \oplus \cdots \oplus \f_k$. Then the curve $C$ is given by the system of equations $\f_i=0$, $i=1, \ldots, k-1$ and the function $g$ is the restriction of $\f_k$ to the curve $C$. Take $p \in \overline{C} \cap O_\rho$ and let $\gamma(u)$ be a local parametrization of $\overline{C}$ in a neighborhood of $p$. Since $\overline{C}$ is transverse to $O_\rho$ at $p$, the order at $u=0$ of the curve $\gamma$ is equal to $\xi$. Now, Theorem \ref{th-f-U-sigma} states that $x^{-\beta_k}\f_k$ is a regular function on the affine chart $U_\rho$ where $\langle \xi, \beta_k \rangle = c_k^\xi$. Moreover, $x^{-\beta_k}\f_k$ does not vanish at $p$. This is because, $\f$ is $\L$-non-degenerate and hence $x^{-\beta_i} \f_i(x) = 0$, $i=1,\ldots,k$, has no solutions on $O_\rho$ (Corollary \ref{cor-system-at-infinity}). Thus, we conclude that the order at $u=0$ of $g(u) = \f_k(\gamma(u))$ is equal to $\langle \xi, \beta_k \rangle = c_k^\xi$, as claimed.
\end{proof}

On the other hand, by Corollary \ref{cor-system-at-infinity}, the points in the intersection $\overline{C} \cap O_\rho$ are given 
by the solutions of the truncated system $\bar{\f}^\xi(x) = 0$ where $\bar{\f}^\xi = \bar{\f}^\xi_1 \oplus \cdots \oplus \bar{\f}^\xi_k$.
From Theorem \ref{th-char-seq-trunc}, we know what the characteristic sequences of the truncations $\bar{\f}^\xi_1, \ldots, \bar{\f}^\xi_{k}$ are. Thus, by the induction hypothesis applied to Theorem \ref{th-BKK-vec2}(b) and the orbit $O_\rho$ (same argument as in Theorem \ref{th-Minkowski-weights}), we have:
\begin{equation}   \label{equ-ind-step}
|\overline{C} \cap O_\rho| = (n-1)!~\MVol_{n-1}(\Delta^\xi_1, \Delta^\xi_2 - \Delta^\xi_1, \ldots, \Delta^\xi_{n-1} - \Delta^\xi_{n-2}).    
\end{equation}
Putting \eqref{equ-sum-poles} and \eqref{equ-ind-step} together and using Lemma \ref{lem-ord-g}, we obtain:
\begin{equation}
|Y(\f)| = (n-1)! \sum_{\rho \in \Sigma(1)} h_n(\xi)~\MVol_{n-1}(\Delta^\xi_1, \Delta^\xi_2 - \Delta^\xi_1, \ldots, \Delta^\xi_{n-1} - \Delta^\xi_{n-2}).   
\end{equation}

The claim now follows from the following known inductive formula for the mixed volume of (virtual) polytopes (see \cite[Section 27.8.2]{Askold-BZ}).
\begin{lemma}  \label{lem-mvol-ind-formula}
Let $P_1, \ldots, P_n$ be virtual polytopes and let $h_n$ denote the support function of $P_n$. Then we have the following inductive formula for the mixed volume of the $P_i$:
$$n!\MVol(P_1, \ldots, P_n) = (n-1)! \sum_{\rho \in \Sigma(1)} h_n(\xi) \MVol(P_1^\xi, \ldots, P_{n-1}^\xi).$$
Here $\Sigma$ is a complete fan where all the support functions of the $P_i$ are piecewise linear with respect to $\Sigma$, and $\xi$ is the primitive generator of a ray $\rho \in \Sigma(1)$. 
\end{lemma}
\end{proof}

\section{Toric vector bundle associated to an invariant subspace}  \label{sec-tvb}
A \emph{toric vector bundle} is a vector bundle $\E$ on a $T$-toric variety $X_\Sigma$ together with a $T$-linearization. As usual let $\L$ be a $T$-invariant subspace of vector-valued Laurent polynomials and let $\Sigma$ be a fan refining the fan $\Sigma_\L$. In this section we naturally correspond a toric vector bundle $\E_{\L, \Sigma}$ on $X_\Sigma$. The rank of the vector bundle $\E_{\L, \Sigma}$ is $r$, the same as the rank of the subspace $\L$ (which is equal to $\dim(E)$).

\subsection{Construction of the toric vector bundle $\E_{\L, \Sigma}$}
We first explain the construction of the toric vector bundle $\E_{\L, \Sigma}$ in terms of transition functions. With notation as in Section \ref{subsec-trunc-infinity}, let $\sigma$ be a full dimensional cone in $\Sigma$ which refines the fan $\Sigma_\L$. Thus, the orbit $O_\sigma$ consists of a single ($T$-fixed) point $x_\sigma$. Let $F^\sigma_\bullet = (F_1^\sigma \subsetneqq \cdots \subsetneqq F^\sigma_k = E)$ be the flag corresponding to any $\xi \in \sigma^\circ$, and $E=V_1 \oplus \cdots \oplus V_k$ a compatible direct sum decomposition, and thus $F_i/F_{i-1} \cong V_i$, for all $i$. 
Recall that, since $\Sigma$ refines $\Sigma_\L$, the multi-valued support function $\h_\L$ is linear on the cone $\sigma$. Suppose the values of $\h_\L$ are given by the characters $\beta_{\sigma, 1}, \ldots, \beta_{\sigma, k}$. That is, for any $\xi \in \sigma^\circ$, the $i$-th critical number $c^\xi_i$ is given by $\langle \xi, \beta_{\sigma, i} \rangle$. 
The above data defines a linear action $\phi_\sigma: T \to \GL(E)$, of torus $T$ on the vector space $E$ by the following property: for each $i$, $V_i$ is the weight space with weight $\beta_{\sigma, i}$. 

For each full dimensional cone $\sigma \in \Sigma$, let $\E_\sigma$ be the trivial bunle $U_\sigma \times E$ on the toric affine chart $U_\sigma$ together with the $T$-action:
$$t \cdot (x, e) := (t \cdot x, \phi_\sigma(t)(e)), ~~\forall t \in T,~\forall x \in U_\sigma,~\forall e \in E.$$

As in \cite[Section 2]{Klyachko}, one proves the following:
\begin{lemma}  \label{lem-trans-function-E-L}
Let $\sigma, \sigma'$ be full dimensional cones in $\Sigma$ with $\tau = \sigma \cap \sigma'$. Then the map $\phi_{\sigma} \phi_{\sigma'}^{-1}: T \to \GL(E)$ extends to a regular map $U_\sigma \to \GL(E)$.    
\end{lemma}

It follows that the $g_{\sigma, \sigma'} := \phi_{\sigma'} \phi_{\sigma}^{-1}$ satisfy the cocycle conditions for the open cover $\{U_\sigma \mid \sigma \in \Sigma(n) \}$ and hence define a toric vector bundle $\E_{\L, \Sigma}$ on the toric variety $X_{\Sigma}$. 


\begin{theorem} \label{th-E_L-sec}
We have the following:
\begin{itemize}
\item[(a)] Any $\f \in \L$ naturally  corresponds to a section of $\E_{\L, \Sigma}$ over the open orbit in $X_\Sigma$ that can be extended to a global regular
section $s_\f$ on $X_\Sigma$.
\item[(b)] The sections $s_\f$, $\f \in \L$, surject to every fiber of $\E_{\L, \Sigma}$, that is, for any $x \in X_{\Sigma}$ and any vector $v$ in the fiber $(\E_{\L, \Sigma})_x$, there is $\f \in \L$ such that  $v=s_\f(x)$.
\end{itemize}
\end{theorem}
\begin{proof}
(a) For $\f \in \L$, define the global section $s_\f \in H^0(X_{\Sigma}, \E_{\L, \Sigma})$ such that, for each full-dimensional cone $\sigma$, on the trivial chart $U_\sigma$ is given by:
$$s_\sigma := x^{-\beta_1} \f^\sigma_1 \oplus \cdots \oplus x^{-\beta_k} \f^\sigma_k \in \k[U_\sigma] \otimes E,$$
(see Section \ref{subsec-trunc-infinity} and Theorem \ref{th-f-U-sigma} for the definitions of $\f_i^\sigma$ and $x^{-\beta_i}\f_i^\sigma$). From the definition of the transition functions $g_{\sigma, \sigma'}$ for the $T$-equivariant bundle $\E_{\L, \Sigma}$ it follows that these $s_\sigma$ glue together to define a global section $s_\f$. Since $\f = \f_1 \oplus \cdots \oplus \f_k$, we see that $\f$ can be recovered from $s_\sigma$. Thus, $\f \mapsto s_\f$ is one-to-one. It follows that $\f \to s_\f$ gives a $T$-equivariant linear embedding of $\L$ into $H^0(X_{\Sigma}, \E_{\L, \Sigma})$, as required.
(b) Since $\f \mapsto s_\f$ is $T$-equivariant, it is enough to show the claim at every $T$-fixed point. Let $\sigma$ be a full dimensional cone with corresponding fixed point $x_\sigma$. From the definitions, for $\f \in \L$, the restriction ${s_\f}_{|x_\sigma} \in \E_{x_\sigma} \cong E$ is given by:
$${s_\f}_{|x_\sigma} = e_1 \oplus \cdots \oplus e_k,$$ where $e_i$ denotes the coefficient of the term $x^{\beta_i}$ in $\f_i$. But since $\sigma$ is full dimenaional, by Lemma \ref{lem-P-i-sigma}(c), we have $E_{\beta_1} + \cdots + E_{\beta_k} = V_1 \oplus \cdots \oplus V_k = E$. This shows that the map $\L \to \E_{x_\sigma} \cong E$ given by $\f \mapsto {s_\f}_{|x_\sigma} = e_1 \oplus \cdots \oplus e_k$ is surjective as required.
\end{proof}

Finally, we look at the connection with Klyachko's classification of toric vector bundles. Toric vector bundles on a toric variety $X_\Sigma$ were classified by Klyachko (see \cite[Section 2]{Klyachko}) in terms of \emph{compatible systems of filtrations} which we now briefly explain. 

Let $\Sigma$ be a fan in $N_\R$ with corresponding toric variety $X_\Sigma$. We fix a point $x_0$ in the open orbit $U_0$ to identify it with the torus $T$. Let $\E$ be a rank $r$ toric vector bundle over $X_\Sigma$ and let $E = \E_{x_0} \cong \k^r$ be the fiber over the distinguished point $x_0$. In \cite[Section 2]{Klyachko}, toric vector bundles $\E$ over $X_\Sigma$ are classified by the data of decreasing filtrations $(\t{E}^\rho_i)_{i \in \Z}$ in $E$, for all rays $\rho \in \Sigma(1)$. One requires that the filtrations $(\t{E}^\rho_i)_{i \in \Z}$ satisfy the following \emph{compatibility conditions}: for any $\sigma \in \Sigma$ there exists a basis $B_\sigma = \{b_{\sigma, 1}, \ldots, b_{\sigma, r}\}$ for $E$ and a multi-set of characters $\beta(\sigma) = \{\beta_{\sigma, 1}, \ldots, \beta_{\sigma, r}\} \subset M$ such that for any ray $\rho \in \sigma(1)$, the corresponding filtration can be recovered from $B_\sigma$ and $\beta(\sigma)$ by:
\begin{equation} \label{equ-Klyachko-comp}
\t{E}^\rho_i = \sum_{\langle v_\rho, \beta_{\sigma, j} \rangle \geq i} \k~b_{\sigma, j},
\end{equation}
where $v_\rho$ denotes the primitive generator of the ray $\rho$.
The filtration spaces $\t{E}^\rho_i$ are constructed as follows. For a ray $\rho \in \Sigma(1)$ let $U_\rho$ be the corresponding affine chart. Take $\alpha \in M$ with $\langle v_\rho, \alpha \rangle = i$. Then the subspace $\t{E}^\rho_i$ is the image of the $\alpha$-weight space $H^0(U_\rho, \E_{|U_\rho})_\alpha$ under the restriction map:
$$H^0(U_\rho, \E_{|U_\rho}) \to E=\E_{x_0}.$$

Recall from Section \ref{subsec-multi-val-supp} that for each $\xi \in N_\R$ we define an increasing $\R$-filtration $(E^\xi_i)_{i \in \R}$ in $E$ given by:
$$E^\xi_i = \sum_{\langle \xi, \alpha \rangle \leq i} E_\alpha.$$
The filtrations $(E^\xi_i)_{i \in \R}$ are used in the construction of the multi-valued support function $\h_\L$. For our purposes, 
it was more convenient to choose the convention, $\leq i$ in the definition of the filtration. Alternatively, we can use $\geq i$ convention to define decreasing filtrations $(\t{E}^\xi_i)_{i \in \R}$:
\begin{equation}  \label{equ-tilde-E}
\t{E}^\xi_i = \sum_{\langle \xi, \alpha \rangle \geq i} E_\alpha.    
\end{equation}
One verifies the following from the defintion of $\E_{\L, \Sigma}$ and the filtrations $(\t{E}^\xi_i)$.
\begin{proposition} \label{prop-E-L-Klyachko}
For $\xi=v_\rho$, the primitive generator of a ray $\rho \in \Sigma$, the filtration $(\t{E}^\xi_i)_{i \in \Z}$ (defined in \eqref{equ-tilde-E}), coincides with the Klyachko filtration associated to $\rho$ for the toric vector bundle $\E_{\L, \Sigma}$. 
\end{proposition}
\begin{proof}
Let $\sigma$ be a maximal cone and $\rho \in \sigma(1)$. The Klyachko data of $\E_{\L, \Sigma}$ is the multi-set of characters $\beta(\sigma) = \{\beta_{\sigma, 1}, \ldots, \beta_{\sigma, k}\}$ (where each $\beta_i$ is repeated $\dim(V_i)$ times), together with any basis $B_\sigma$ for $E$ obtained by taking a union of bases for the subspaces $V_i$. One then verifies that the equation \eqref{equ-Klyachko-comp} defines the same filtration as \eqref{equ-tilde-E}.
\end{proof}

\subsection{Equivariant Chern roots}
We recall (see \cite{Payne-moduli}) that the $i$-th $T$-equivariant Chern class of a toric vector bundle $\E$ on the toric variety $X_\Sigma$ is represented by the piecewise polynomial function $c_i^T(\E)$ whose restriction to any cone $\sigma \in \Sigma$ is the $i$-th elementary symmetric function on the $\beta_{\sigma, i}$, where $\beta(\sigma) = \{\beta_{\sigma, 1}, \ldots, \beta_{\sigma, r}\}$ is the multi-set of characters associated to $\sigma \in \Sigma$ in the Klyachko compatibility conditions. Applying this to the toric vector bundle $\E_{\L, \Sigma}$ we obtain the following. 

\begin{proposition}  \label{prop-equiv-Chern-roots}
Let $\L \subset E \otimes \k[T]$ be an invariant subspace of vector-valued Laurent polynomials. Let $\h_L = (h_1, \ldots, h_r)$ be the corresponding multi-valued support function as in Section \ref{subsec-multi-val-supp}. For any $i=1, \ldots, r$, the $i$-th equivariant Chern class of the toric vector bundle $\E_{\L, \Sigma}$ is represented by the $i$-th elementary symmetric function on $h_1, \ldots, h_r$. In particular, the top equivariant Chern class of $\E_{\L, \Sigma}$ is represented by the degree $r$ piecewise polynomial function $h_1 \cdots h_r$.
In other words, $\{h_1, \ldots, h_r\}$ can be viewed as a choice of equivariant Chern roots for $\E_{\L, \Sigma}$. 
\end{proposition}

Finally, we have the following alternative version of our BKK theorem.
\begin{theorem}[Alternative version of vector-valued BKK theorem] \label{th-alt-vec-BKK}
Let $r=n$ and let $\L$ be an invariant subspace of rank $n$. Let $s \in H^0(X_{\Sigma}, \E_{\L, \Sigma})$ be a generic section. Then the number of points $x \in X_{\Sigma}$ such that $s(x) = 0$ is given by $n! \MVol(\h_\L)$.
\end{theorem}
\begin{proof}
The claim follows from the next lemma applied to the piecewise linear functions $h_1, \ldots, h_r$ in Proposition \ref{prop-equiv-Chern-roots}. 
\begin{lemma}  \label{lem-equiv-class-tv}
Let $a \in A^n(X_\Sigma, \Q)$ be a top equivariant Chow cohomology class. Let $h_i$ be piecewise linear functions with corresponding virtual polytopes $P_i$ such that the cohomology class $a$ is represented by the piecewise polynomial function $h_1 \cdots h_n$. Then the degree of $a$ is given by $n! \MVol_n(P_1, \ldots, P_n)$.    
\end{lemma}
\end{proof}

\subsection{Maps to Grassmannian}  \label{subsec-map-Grassmannian}
Let $E$ be an $r$-dimensional $\k$-vector space 
and let $\L \subset E \otimes \k[T]$ be a $T$-invariant subspace. Recall that, without loss of generality, we assume: $$\sum_\alpha E_\alpha = E.$$ 

Let $N = \dim \L$. We can then define a morphism $\Phi_\L: T \to \Gr(\L, N-r) = \Gr(\L^*, r)$ as follows:
For $x \in T$ consider the evaluation map $\ev_x: \L \to E$:
$$\sum_{\alpha} e_\alpha \otimes x^\alpha \mapsto \sum_{\alpha} e_\alpha x^\alpha.$$
Define $\Phi_\L(x)$ to be the kernel of the evaluation map $\ev_x$, that is:
$$\Phi_\L(x) = \{ \sum_{\alpha \in \A} e_\alpha \otimes x^\alpha \mid \sum_{\alpha \in \A} e_\alpha x^\alpha = 0 \}.$$

Since $\sum_\alpha E_\alpha = E$ it follows that $\dim \ker(\ev_x) = N-r$ and thus $\Phi_\L$ is well-defined. 

The following is straightforward.
\begin{proposition}
The Kodaira map $\Phi_\L: T \to \Gr(\L, N-r)$ is $T$-equivariant.
\end{proposition}

Let $Y_\L$ be the closure of the image of $\Phi_\L$ in $\Gr(\L^*, r)$. It is a an irreducible (possibly non-normal) $T$-variety with an open $T$-orbit. Hence if $\Phi_\L$ is one-to-one then the normalization of $Y_L$ is a $T$-toric variety.


\begin{remark} 
One can show that the map $\Phi_\L: T \to \Gr(\L, N-r)$ extends to a morphism on the whole toric variety $X_{\Sigma_\L}$. Moreover, the pull-back of the tautological subbundle on the Grassmannian to $X_\Sigma$ is $\E_{\L, \Sigma}$.    
\end{remark}

\section{Example: hyperplane arrangements}  \label{sec-hyperplane-arrangement}
In this section we consider an important special case of our results in Section \ref{sec-BKK}, namely, the case of a hyperplane arrangement. As mentioned in the introduction, our results are equivalent to a computation known in the tropical geometry literature for the class of a linear space in the ring of conditions of $T$ as a product of virtual polytopes associated to matroids (see \cite{Huh} and further discussions in \cite[Appendix III]{BEST}).

Let $u_0, \ldots, u_N$ be nonzero linear forms on the vector space $\k^{n+1}$. They define a hyperplane arrangement $\H = \{H_0, \ldots, H_N\}$ in $\mathbb{P}^n$, where $H_i$ is the (projective) hyperplane defined by $u_i = 0$. Without loss of generality we assume that the $u_i$ span the dual space $(\k^{n+1})^*$. Equivalently, $\bigcap_{i=0}^N H_i = \emptyset$.

Let $\mathbb{T}^N$ denote the $N$-dimensional quotient torus $(\k^*)^{N+1} / (\k^*)_{\diag}$, where $(\k^*)_{\diag}$ is the subgroup $\{(t, \ldots, t) \mid t \in \k^* \}$.
Recall that for a finite subset $\A \subset \Z^{N+1}$ we let $L_\A = \{f(z) = \sum_{\alpha \in \A} c_\alpha z^\alpha \mid c_\alpha \in \k\}$.
As in the classic BKK theorem, consider finite subsets of characters $\A_1, \ldots, \A_n \subset \Z^{N+1}$ that lie in the hyperplane defined by the sum of coordinates equal to $0$. The $\A_i$ can be thought of as subsets in the character lattice of the torus $\mathbb{T}^N$. 
Let $f_i \in L_{\A_i}$, $i=1, \ldots, n$, be generic Laurent polynomials (note that by assumption each $f_i$ is homogeneous of degree $0$). Let $X = \mathbb{P}^n \setminus \bigcup_{i=0}^N H_i$ be the complement of the hyperplane arrangement. In this section, we consider the following BKK type problem:

\begin{problem}   \label{prob-hyperplane-Laurent}
Give a mixed volume formula for the number of solutions $x \in X$ of the system of Laurent polynomial equations:
\begin{equation}  \label{equ-hyperplane-Laurent-system}
f_1(u_0(x), \ldots, u_N(x)) = \cdots = f_n(u_0(x), \ldots, u_N(x)) = 0.    
\end{equation}    
\end{problem}

First we consider the case where the hyperplane arrangement is in general position, in the sense that the intersection of any collection of $n+1$ hyperplanes in $\H$ is empty. In this case, it can be shown that the number of solutions of \eqref{equ-hyperplane-Laurent-system} is equal to 
\begin{equation}  \label{equ-MVol-hyperplane-generic}
N! \MVol_N(\Delta, \ldots, \Delta, P_1, \ldots, P_n),    
\end{equation}
where the mixed volume, denotes the mixed volume in the hyperplane the sum of coordinates equal to $0$. Also, $\Delta$ is the standard simplex in the hyperplane the sum of coordinates equal to $0$, and is repeated $N-n$ times in the mixed volume, and $P_i = \conv(\A_i)$, $i=1, \ldots, n$. 

Below we show that the general case of Problem \ref{prob-hyperplane-Laurent} (when the hyperplanes are not necessarily in general position) follows from Theorem \ref{th-main}. To this end, we assign an invariant subspace of vector-valued Laurent polynomials to the hyperplane arrangement $\H$.
Consider the map $X \to \mathbb{T}^N$ given by:
$$[v] \mapsto (u_0(v): \cdots :u_N(v)),$$ where $[v] \in \mathbb{P}^n$ denotes the point represented by $v \in \k^{n+1} \setminus \{0\}$. This gives an isomorphism of $X$ with a linear subvariety $V$ obtained by intersecting an $n$-dimensional (projective) plane with the torus $\mathbb{T}^N$. Let $E \cong \k^{N-n}$ be the dual space to the kernel of the linear map which sends $z_i$ to $u_i$, $i=0, \ldots, N$. Also let $\ell_i \in E$ be the projection on the $i$-th coordinate restricted to this kernel. Then $\{\ell_0, \ldots, \ell_N\}$
is a spanning set for $E$.

\begin{remark}  \label{rem-dual-matroid}
One verifies that the matroid structure of $\{\ell_0, \ldots, \ell_N\}$ is the dual matroid to that of $\{u_0, \ldots, u_N\}$. That is, a subset of $\{\ell_0, \ldots, \ell_N\}$ is independent if and only if the complement of the corresponding subset in $\{u_0, \ldots, u_N\}$ is independent.    
\end{remark}

The invariant subspace $\L_\H$ associated to $\H$ is:
$$\L_\H = \bigoplus_{i=0}^N ~\langle \ell_i \rangle \otimes z_i \subset E \otimes \k[z_0^\pm, \ldots, z_N^\pm].$$

From the definition one verifies the following:
\begin{proposition}  \label{prop-char-seq-lin-space}
The characteristic sequence of polytopes $(\Delta_1, \ldots, \Delta_{N-n})$ for the subspace $\L_\H$ is given by:
\begin{align*}
\Delta_i &= \conv\{\e_{j_1} + \cdots + \e_{j_i} \mid \ell_{j_1}, \ldots, \ell_{j_i} \text{ are linearly independent} \} \\ &= \conv\{\e_{j_1} + \cdots + \e_{j_i} \mid \bigcap_{j \notin \{j_1, \ldots, j_i\}} H_j = \emptyset  \}, 
\end{align*}
for $i=1, \ldots, N-n$. Here $\{\e_0, \ldots, \e_N\}$ denotes the standard basis for $\Z^{N+1}$. As usual, we put $\Delta_0 = \{0\}$.
\end{proposition}

As above, let $V$ be the image of $X$ in $\mathbb{T}^N$. It is a linear subvariety defined by the ideal of relations among the vectors $u_i$. Theorem \ref{th-intro-ring-of-conditions} applied to the invariant subspace $\L_\H$ gives us the following.

\begin{corollary}[The class of a linear space in the ring of conditions]  \label{cor-class-lin-space]}
The class of $V$ in the ring of conditions of the torus $\mathbb{T}^N$ is given by $\prod_{i=1}^{N-n} [\Delta_i - \Delta_{i-1}] $ \end{corollary}
\begin{proof}
Let $\f = \sum_i \ell_i \otimes z_i$. Then $V$ is defined by the vector equation $\f = 0$.
We note that, in the case of $\L_\H$, each subspace $E_\alpha$ is $1$-dimensional and hence the genericity condition for $\f$ in Theorem \ref{th-BKK-vec2} (also Theorem \ref{th-intro-ring-of-conditions}) is automatically satisfied. The claim now follows from Theorem \ref{th-BKK-vec2}. 
\end{proof}

Now we are ready to give an answer to Problem \ref{prob-hyperplane-Laurent}. 

\begin{corollary}[BKK with homogeneous linear coordinates]
\label{cor-BKK-hyperplane}
With notation as before, 
let $f_i \in L_{\A_i}$, $i=1, \ldots, n$, be generic homogeneous degree $0$ Laurent polynomials in $z=(z_0, \ldots, z_N)$ with fixed supports $\A_i \subset \Z^{N+1}$ that lie in the hyperplane given by the sum of coordinates equal to $0$. Then the number of solutions $x \in X = \mathbb{P}^n \setminus \bigcup_{i=0}^N H_i$ of  the system $f_1(u_0(x), \ldots, u_N(x)) = \cdots = f_n(u_0(x), \ldots, u_N(x)) = 0$ is equal to:
$$N! \MVol(\Delta_1, \Delta_2 - \Delta_1, \ldots, \Delta_{N-n} - \Delta_{N-n-1}, P_1, \ldots, P_n),$$
where $P_i = \conv(\A_i)$, $i=1, \ldots, n$. Note that strictly speaking, the polytopes $P_i$ and $\Delta_i$ do not lie in the same hyperplane, but rather they all lie in parallel hyperplanes defined by the sum of coordinates equal to constants. The mixed volume above, denotes the mixed volume in the hyperplane the sum of coordinates equal to $0$, and after translating the polytopes to this hyperplane.
\end{corollary}

Finally we recover an Alexandrov-Fenchel inequality related to a linear subspace. This result is not new and has been known (see \cites{Huh, BEST}). 

\begin{corollary}[An Alexandrov-Fenchel inequality for the class of a linear subspace]  \label{cor-AF-hyperplane}
Let $P_1, P_2, P_3, \ldots, P_n$ be convex bodies in $\R^{N+1}$ that are parallel to the hyperplane defined by sum of coordinates equal to $0$. Let $\Delta_1, \ldots, \Delta_{N-n}$ be convex polytopes as defined above. Then we have:
$$\MVol_N(\Delta_1, \ldots, \Delta_{N-n}-\Delta_{N-n-1}, P_1, P_1, \ldots, P_n) \MVol_N(\Delta_1, \ldots, \Delta_{N-n}-\Delta_{N-n-1}, P_2, P_2, \ldots, P_n)$$ 
$$\leq \MVol_N(\Delta_1, \ldots, \Delta_{N-n}-\Delta_{N-n-1}, P_1, P_2, \ldots, P_n)^2.$$
\end{corollary}
\begin{proof}
This is a consequence of Corollary \ref{cor-BKK-hyperplane} and the Khovanskii-Teissier inequality applied to the irreducible variety $X$ (see also \cite[Part IV]{KKh-Annals}).
\end{proof}

\section{An Alexandrov-Fenchel inequality for vector-valued Laurent polynomials}  \label{sec-AF}
Let $\L=\oplus_{\alpha \in \A} E_\alpha \otimes x^\alpha \subset E \otimes \k[T]$ be an invariant subspace of vector-valued Laurent polynomials of full rank $r=n$. Let $\f \in \L$ and consider the equation:
\begin{equation*}   \label{equ-Y-f}
\f(x) = \sum_{\alpha \in \A} e_\alpha x^\alpha = 0.    
\end{equation*}
For each $\alpha \in \A$, fix a basis $B_\alpha = \{b_{\alpha, 1}, \ldots, b_{\alpha, r_\alpha}\}$ for $E_\alpha$ and write $e_\alpha = \sum_i c_{\alpha, i} b_{\alpha, i}$. For each $b_{\alpha, i}$ let us introduce an auxiliary scalar variable $z_{\alpha, i}$ and put $z = (z_{\alpha, i})_{i, \alpha} \in \mathbb{T}^N = (\k^*)^N$ where $N=\dim(\L) = \sum_\alpha r_\alpha$. It is obvious, that the solutions $x \in T$ of the system $\f(x) = 0$ are in one-to-one correspondence with the solutions $(x, z) \in T \times \mathbb{T}^N$ of the system of vector equations:
\begin{align*}
\sum_{\alpha} \sum_{i=1}^{r_\alpha} b_{\alpha, i} z_{\alpha, i}x^{\alpha} &= 0 \\
z_{\alpha, i} - c_{\alpha, i} &= 0, \quad \forall \alpha, i\\
\end{align*}

In the torus $T \times \mathbb{T}^N$, consider the change of coordinates $w_{\alpha, i} = z_{\alpha, i} x^{-\alpha}$, $\forall \alpha, i$, and $x=x$. The following is then obvious:
\begin{lemma}    \label{lem-aux-variable}.
The number of solutions $x \in T$ of the system $\f(x) = 0$ is equal to the number of solutions $(x, w) \in T \times \mathbb{T}^N$ of the system of vector equations:
\begin{align*}
\sum_{\alpha} \sum_{i=1}^{r_\alpha} b_{\alpha, i} w_{\alpha, i} &= 0 \\
w_{\alpha, i} x^{\alpha} - c_{\alpha, i} &= 0, \quad \forall \alpha, i\\
\end{align*}
\end{lemma}

But the above is a system of Laurent polynomials $w_{\alpha, i} x^{\alpha} - c_{\alpha, i} = 0$, $\forall \alpha, i$, restricted to the linear subspace defined by $\sum_{\alpha} \sum_{i=1}^{r_\alpha} b_{\alpha, i} w_{\alpha, i} = 0$. This is exactly the situation in Corollary \ref{cor-AF-hyperplane}. Thus, from Corollary \ref{cor-AF-hyperplane} we obtain the following:

\begin{corollary}   \label{cor-vec-AF-v2} 
Let $\L$ be an invariant subspace where $r \leq n$. Let $(\Delta_1, \ldots, \Delta_r)$ be the characteristic sequence of polytopes associated to $\L$. Also let $P_1, \ldots, P_{n-r}$ be arbitrary convex bodies in $\R^n$. Then we have:
$$\MVol_n(\Delta_1, \ldots, \Delta_{r}-\Delta_{r-1}, P_1, P_1, \ldots, P_{n-r}) \MVol_n(\Delta_1, \ldots, \Delta_{r}-\Delta_{r-1}, P_2, P_2, \ldots, P_{n-r}))$$ $$\leq \MVol_n(\Delta_1, \ldots, \Delta_{r}-\Delta_{r-1}, P_1, P_2, \ldots, P_{n-r}))^2.$$    
\end{corollary}

More generally, we have:
\begin{corollary}  \label{cor-vec-AF-v1}
Let $\L_1$, $\L_2$ be invariant subspaces of rank $1$ and $\L_3$ an invariant subspace of rank $n-2$. We then have:
\begin{equation}
\MVol(\h_{\L_1 \oplus \L_2 \oplus \L_3}) \geq \MVol(\h_{\L_1 \oplus \L_1 \oplus \L_3}) \MVol(\h_{\L_2 \oplus \L_2 \oplus \L_3}).
\end{equation}
\end{corollary}

\section{Extensions to polymatroids}
\label{sec-polymatroid}
The assignment of a $T$-invariant subspace $\L=\bigoplus_{\alpha\in \mathcal{A}}E_\alpha$ to the characteristic sequence of polytopes $\Delta_1,\Delta_2,\ldots$ and the associated multi-linear function $h_\L$ depends only on a finite amount of combinatorial data associated to the subspace arrangement $\{E_\alpha \mid  \alpha\in \mathcal{A}\}$ known as the \emph{polymatroid} of the subspace arrangement.

In this section we will show how to extend the results from the previous sections to arbitrary polymatroids, which include those that may not come from a subspace arrangement. In particular, using the Hodge theory of matroids developed in \cite{AHK}, we will show that the Alexandrov-Fenchel inequality for vector-valued Laurent polynomials has a natural extension to arbitrary polymatroids. We do this by carrying out essentially a tropical analogue of the argument from the previous section, heavily inspired by arguments from \cite{BEST}.

\subsection{An Alexandrov-Fenchel inequality for polymatroids}  \label{subsec-AF-polymatroid}
\begin{definition}
    A \emph{polymatroid} $\P$ on a finite ground set $\mathcal{A}$ of rank $r$ is a function
    $$\dim:\{\text{subsets of $\mathcal{A}$}\}\to \{0,1,2,\ldots\}$$
    such that $\dim \emptyset=0$,  $\dim (\mathcal{A})=r$, and \begin{enumerate}
        \item $\dim(S)+\dim(T)\ge \dim(S\cup T)+\dim(S\cap T)$, and
        \item whenever $B\subset C$ we have $\dim(B)\le \dim(C)$.
    \end{enumerate} 
    If $\dim S\le |S|$ for all $S$, then $\P$ is called a \emph{matroid}. Given a subspace arrangement $\{E_\alpha \mid \alpha\in \mathcal{A}\}$, the \emph{associated polymatroid} is $S\mapsto \dim \sum_{\alpha \in S} E_\alpha$. This is a matroid if and only if all $E_\alpha$ have dimension at most $1$.
\end{definition}

\begin{definition}
    We define the set of admissible sequences $\mathcal{A}^i$ of size $0\le i \le r$ for $\P$ by
$$\mathcal{A}^i:=\{(\alpha_1,\ldots,\alpha_i):|\{i\mid  \alpha_i\in S\}| \le \dim S\text{ for all $S$}\}.$$
For a matroid the admissible sequences are called the \emph{independent sets}, and the independent sets of size $r$ are called the \emph{bases}.
If $\mathcal{A}\subset \mathbb{Z}^n$ is a set of characters then we can associate the \emph{characteristic polytopes}
$\Delta_1,\Delta_2,\ldots,\Delta_r$ in the same way
$$\Delta_i = \conv\{\sum_{j=1}^i \alpha_j \mid (\alpha_1, \ldots, \alpha_i) \in \A^i \text{ is admissible for } \P \}$$ we would for a subspace arrangement, and we associate in the same way the multi-valued support function $\h_\P=(h_1,\ldots,h_r)$ where $h_i$ is the support function of the virtual polytope $\Delta_i-\Delta_{i-1}$.
\end{definition}
By Rado's theorem, the definition of $\mathcal{A}^i$ agrees with the earlier definition for a subspace arrangement, and we recover our previous definitions. We are now ready to state the analogue of the Alexandrov-Fenchel inequality for polymatroids.
\begin{theorem}
\label{thm-pol-AF-v2}
Let $\P$ be a rank $r$ polymatroid on ground set $\mathcal{A}\subset \Z^n$. Then the statement of Corollary~\ref{cor-vec-AF-v2} holds for the characteristic sequence of polytopes $\Delta_1,\ldots,\Delta_r$ of $\P$. Furthermore, all three terms appearing in the inequality are nonnegative.
\end{theorem}
We note that
    if $\P$ is not represented by a subspace arrangement then the terms in the inequality are not counting solutions to actual systems of vector-valued Laurent polynomial equations. 
    
To prove Theorem~\ref{thm-pol-AF-v2} we will closely follow the steps used to prove Corollary~\ref{cor-vec-AF-v2}. The first step is to emulate choosing a basis $\{b_{i,\alpha}\}$ for each $E_\alpha$. The combinatorial analogue of this is via the \emph{natural matroid} $\mathcal{M}(\P)$ of $\P$ \cite{Helgason} (see also \cite{BCF23}).
\begin{definition} 
   Given a polymatroid $\P$ on ground set $\A$, the \emph{natural matroid} $\mathcal{M}(\P)$ is the matroid whose ground set is $$\widetilde{\mathcal{A}}:=\bigsqcup_{\alpha\in \A} \{\alpha^{(1)},\ldots,\alpha^{(\dim \{\alpha\})}\},$$
   obtained by replacing each $\alpha$ with $\dim \{\alpha\}$-copies of it,
   and whose rank function takes a subset $\widetilde{S} \subset \widetilde{A}$ and associates to it the maximal $i$ for which $\widetilde{S}$ contains all the elements of an admissible sequence in $\mathcal{A}^i$.
\end{definition}
\begin{definition}
    For a matroid $\M$ whose ground set is a finite multiset $\widetilde{\A}$ of $\mathbb{Z}^n$, we define the associated characteristic polytopes $$\Delta_{i}:=\conv\{\sum_{j=1}^i \alpha_i\mid \exists a_j \text{ such that }  (\alpha_1^{(a_1)},\ldots,\alpha_i^{(a_i)})\in \widetilde{\mathcal{A}}^i \},$$
    and we define the associated piecewise linear function $h_\M=(h_1,\ldots,h_r)$ where $h_i$ is the support function for the virtual polytope $\Delta_i-\Delta_{i-1}$.
\end{definition}
The significance of the natural matroid in this context comes from the starightforward observation that
$$\Delta_{i,\P}=\Delta_{i,\M(\P)}\text{ for }0\le i \le r\text{ and }h_\P=h_{\M(\P)}.$$

We now describe the analogue of introducing auxilliary variables to our system of equations.

\begin{definition}
    Define the \emph{base polytope} of $\M$ on ground set $\widetilde{\mathcal{A}}$ to be  $$BP_{i,\M}=\conv\{\sum_{j=1}^i \e_{\beta} \mid (\beta_1,\ldots,\beta_r)\text{ a basis of }\M\}\subset  \mathbb{R}^{|\widetilde{A}|},$$ where $\e_\alpha$ are the standard basis vectors, and let $\pi_{\widetilde{A}}:\mathbb{R}^{|\widetilde{A}|}\to \mathbb{R}^n$ be the map taking $\e_{\alpha_i^{(j)}}$ to $\alpha_i$. Define the multi-valued support function $h_{\M,BP}:=(h_1, \ldots, h_r)$ where $h_1=h_{BP_{1,\M}}$ and $h_i=h_{BP_{i,\M}}-h_{BP_{i-1,\M}}$ for $2\le i \le r$.
\end{definition}
We note that when $\mathcal{M}$ is realizable, $h_{\M,BP}$ is exactly the same as the multi-valued support function which induces the class of the associated linear space in $\R^{|\widetilde{A}|}$. 
\begin{proposition}
\label{prop-usefulvariable}
    Let $\M$ be a matroid on ground set $\widetilde{\mathcal{A}}$, and let $k=|\widetilde{\mathcal{A}}|$. Then
    $$\MVol_n(h_\M,P_1,\ldots,P_{n-r})=\MVol_{n+k}(h_{\M,BP}',\{0\}^n\times P_1,\ldots, \{0\}^n\times P_{n-r},T_1,\ldots,T_k),$$
    where $h_{\M,BP}'(x_1,\ldots,x_n,y_1,\ldots,y_k)=h_{\M,BP}(x_1,\ldots,x_n)$, and $T_i$ is the line segment connecting the standard basis vector $\{0\}^n\times \e_i$ to $\{\pi_{\widetilde{A}}(\e_i)\}\times \{0\}^k$.
\end{proposition}
\begin{proof}
First, note that $\pi_{\widetilde{A}}(BP_{i,\M})=\Delta_{i,\M}$. Therefore by multilinearity of mixed volume, it suffices to show that if $L:\mathbb{Z}^k\to \mathbb{Z}^n$ is a linear map, and $R_{1},\ldots,R_r\subset \mathbb{R}^{n}$ a family of integral polytopes, and $P_1,\ldots,P_{n-r}\subset \R^k$ a second family of integral polytopes, then
    \begin{align*}\MVol_n(L(R_1),\ldots,L(R_r),P_1,\ldots,P_{n-r})\\=\MVol_{n+k}(P_1\times \{0\}^n,\ldots,P_r\times \{0\}^n,\{0\}^k\times R_1,\ldots,\{0\}^k\times R_{n-r},T_1,\ldots,T_k)\end{align*}
    where $T_i$ is the line segment connecting $\e_i\times \{0\}^k$ to $\{0\}\times \{L(\e_i)\}$.

    By the BKK theorem, the right hand side is the number of solutions to the system of equations
\begin{align*}
f_1(x_1,\ldots,x_k)=\cdots =f_r(x_1,\ldots,x_k)=0\\
g_1(z_1,\ldots,z_n)=\cdots =
g_{n-r}(z_1,\ldots,z_n)=0\\
x_1=c_1z^{L(\e_1)},\ldots, x_k=c_kz^{L(\e_k)}
\end{align*}
where $f_i(x_1,\ldots,x_k)$ is a generic Laurent polynomial with $\supp(f_i)=P_i$, $g_i(z_1,\ldots,z_n)$ is a generic Laurent polynomial with $\supp(g_i)=R_i$, and $c_i$ are generic constants. The third set of equations determines $x_1,\ldots,x_k$ and we can rewrite the system as
\begin{align*}
f_1(c_1z^{L(\e_1)},\ldots,c_kz^{L(\e_k)})=\cdots = f_r(c_1z^{L(\e_1)},\ldots,c_kz^{L(\e_k)})&=0\\
g_1(z_1,\ldots,z_n)=\cdots= g_{n-r}(z_1,\ldots,z_n)&=0
\end{align*}
which then computes the left hand side again by the BKK theorem.
\end{proof}

We are now ready to prove Theorem~\ref{thm-pol-AF-v2} using Hodge theory for matroids as developed in \cite{AHK}.

\begin{proof}[Proof of Theorem~\ref{thm-pol-AF-v2}]
With notation as in Proposition~\ref{prop-usefulvariable}, we want to show that $b^2\ge ac$ where
\begin{align*}a&=\MVol_{n+k}(h_{\M,BP}',\{0\}^n\times P_1,\{0\}^n\times P_1,\{0\}^n\times P_3\ldots, \{0\}^n\times P_{n-r},T_1,\ldots,T_k)\\
b&=\MVol_{n+k}(h_{\M,BP}',\{0\}^n\times \{0\}^n\times P_1,\{0\}^n\times P_2,\{0\}^n\times P_3\ldots, \{0\}^n\times P_{n-r},T_1,\ldots,T_k)\\
c&=\MVol_{n+k}(h_{\M,BP}',\{0\}^n\times \{0\}^n\times P_2,\{0\}^n\times P_2,\{0\}^n\times P_3\ldots, \{0\}^n\times P_{n-r},T_1,\ldots,T_k).\end{align*}
    We note at this point that if $\M$ is a realizable matroid then as before $h'_{BP,\M}$ is the class of a linear space in $(\mathbb{C}^*)^{n+k}$ (more precisely the intersection with $(\mathbb{C}^*)^{n+k}$ of the product $L_\M\times \mathbb{C}^n$ of the linear space associated to $\M$ and $\mathbb{C}^n$), and we may conclude as in the proof of Corollary~\ref{cor-AF-hyperplane} we may conclude by applying the Khovanskii-Teissier inequality.
    
    If $\M$ is not a linear space closure, then we must proceed using the Hodge theory of matroids as developed in \cite{AHK}, following the same ideas as in \cite{BEST}. Recall that the dual Bergman fan  $\Sigma_{\M^{\perp}}\subset \mathbb{R}^{n}/\langle (1,\ldots,1)\rangle$ is a Lorentzian fan, and in the ring of conditions of $(\mathbb{C}^*)^{n}/\mathbb{C}^*$ we have $$[\Sigma_{M^{\perp}}]=\prod_{j=1}^r(h_{BP_{i,\M}}'(x_1-x_n,\ldots,x_{n-1}-x_n,0)-h_{BP_{i-1,\M}}'(x_1-x_n,\ldots,x_{n-1}-x_n,0)).$$ Then the class in the ring of conditions of $(\C^*)^{n+k}$ associated to $h'_{BP,\M}$ is the pullback of $[\Sigma_{M^{\perp}}]$ under the projection $(\R)^{n+k}\to (\R)^n/\langle (1,\ldots,1)\rangle$ given by $(x_1,\ldots,x_n,y_1,\ldots,y_k)\mapsto (x_1,\ldots,x_n)+\R\langle (1,\ldots,1)\rangle$, which is abstractly the tropical fan associated to $\Sigma_M$ times $\R^{k+1}$. Choosing a splitting $\mathbb{Z}^{n+k}=\{\sum_{i=1}^n x_i=0\}\oplus M$ where $M\subset \mathbb{Z}^{n+k}$ is a complementary sublattice of rank $k+1$, we see that there is a matrix $g\in GL_n(\Z)$ such that that the class of $h'_{BP,\M}$ in the ring of conditions of $(\C^*)^{n+k}$ is given by $[g\cdot (\Sigma_{M^{\perp}}\times \Sigma')]$, where $\Sigma'$ is any complete full dimensional smooth fan in $M_{\R}$. Let $\widetilde{\Sigma}$ be a smooth fan refining $g\cdot (\Sigma_{M^{\perp}}\times \Sigma')$ as well as the normal fans of all $\{0\}^n\times P_i$ and $T_i$. Then because $\Sigma_{M^{\perp}}$ is Lorentzian, any complete smooth fan is Lorentzian, products of Lorentzian fans are Lorentzian, translations of Lorentzian fans by invertible integral matrices, and smooth refinements of Lorentzian fans are Lorentzian, we conclude that $\widetilde{\Sigma}$ is Lorentzian, and by construction supports the convex piecewise-linear support functions associated to all $\{0\}^n\times P_i$ and $T_1,\ldots,T_k$. This implies that $a,b,c\ge 0$ and the Alexandrov-Fenchel inequality $b^2\ge ac$ holds, since the three terms may be viewed as intersection numbers directly on $\widetilde{\Sigma}$.
\end{proof}


\begin{thebibliography}{99}
\bibitem[AHK18]{AHK} Adiprasito, K.; Huh, H.; Katz, E. \textit{Hodge theory for combinatorial geometries}. Annals of Mathematics 188, no. 2 (2018): 381--452.
\bibitem[ADH23]{ADH} Ardila, F.; Denham, G.; Huh, J. {\it Lagrangian geometry of matroids}. Journal of the American Mathematical Society 36 (2023), 727--794.
\bibitem[Batyrev94]{Batyrev94} Batyrev, V. V. \textit{Dual polyhedra and mirror symmetry for Calabi-Yau hypersurfaces in toric varieties}. Journal of Algebraic Geometry, 3(3), 493--535, 1994.
\bibitem[BEST23]{BEST} Berget, A.; Eur, C.; Spink, H.; Tseng, D. {\it Tautological classes of matroids}. Invent. Math.233(2023), no.2, 951--1039.
\bibitem[BKK76]{BKK} Bernstein, D. N.; Kushnirenko, A. G.; Khovanskii, A. G.
\bibitem[BCF23]{BCF23} Bonin, J. E.; Chun, C.; Fife, T. {\it The natural matroid of an integer polymatroid} SIAM J. Discrete Math. vol. 37, no. 3, 1751--1770
\textit{Newton polyhedra}.
Uspehi Mat. Nauk 31(1976), no.3, 201--202.
\bibitem[Brion94]{Brion} Brion, M.
\textit{Piecewise polynomial functions, convex polytopes and enumerative geometry}. Parameter spaces (Warsaw, 1994), 25--44.
Banach Center Publ., 36. Polish Academy of Sciences, Institute of Mathematics, Warsaw, 1996
\bibitem[DaKh87]{Danilov-Khovanskii} Danilov, V. I.; Khovanskii, A. G. {\it Newton polyhedra and an algorithm for computing Hodge–Deligne numbers}. Mathematics of the USSR-Izvestiya, 1987, Volume 29, issue 2, p. 279--298
\bibitem[DP83]{DeConcini-Procesi} De Concini, C.; Procesi, C. \textit{Complete symmetric varieties II. Intersection theory}. Algebraic groups and related topics (Kyoto/Nagoya, 1983), 481--513.
Adv. Stud. Pure Math., 6. North-Holland Publishing Co., Amsterdam, 1985

\bibitem[Esterov1]{E1} Esterov, A. \textit{Sch\"on complete intersections}. arXiv:2401.12090
\bibitem[Esterov2]{E2} Esterov, A. \textit{Engineered complete intersections: slightly degenerate Bernstein--Kouchnirenko--Khovanskii}. arXiv:2401.12099
\bibitem[Esterov3]{E3} Esterov, A. \textit{Engineered complete intersections: eliminating variables and understanding topology}. arXiv:2504.16018

\bibitem[EFLS24]{EFLS} Eur, C.; Fink, A.; Larson, M.; Spink, H. \textit{Signed permutohedra, delta-matroids, and beyond}. Proceedings of the London Mathematical Society, 128(3), (2024).

\bibitem[EHL23]{EHL} Eur, C.; Huh, J.; Larson, M. \textit{Stellahedral geometry of matroids}. Forum of Mathematics Pi, Paper No. e24, 48 (2023).

\bibitem[EKKh21]{EKKh} Kazarnovskii, B. Ya; Khovanskii, A. G.; Esterov, A. I. {\it Newton polytopes and tropical geometry}. Uspekhi Mat. Nauk 76:1(457) (2021), 95–190; translation in Russian Math. Surveys 76:1, 91--175.
\bibitem[FSt97]{Fulton-Sturmfels} Fulton, W.; Sturmfels, B. \textit{Intersection theory on toric varieties}. Topology 36 (1997), no.2, 335--353.
\bibitem[H72]{Helgason} T. Helgason, Aspects of the theory of hypermatroids, in {\it Hypergraph Seminar (Proc. First Working Sem., Ohio State Univ., Columbus, Ohio, 1972; dedicated to Arnold Ross)}, pp. 191--213, Lecture Notes in Math., Vol. 411, Springer, Berlin-New York
\bibitem[Huh]{Huh} 
Huh, J. {\it Rota’s conjecture and positivity of algebraic cycles in permutohedral varieties}. Ph.D. thesis
(2014)
\bibitem[KKh12]{KKh-Annals} Kaveh, K.; Khovanskii, A. G. \textit{Newton-Okounkov bodies, semigroups of integral points, graded algebras and intersection theory}. Ann. of Math. (2) 176 (2012), no.2, 925--978.
\bibitem[KKh]{KKh-survey} Kaveh, K.; Khovanskii, A. G. \textit{A short survey on Newton polytopes, tropical geometry and ring of conditions of algebraic torus}. arXiv:1803.07001
\bibitem[KM25]{Kaveh-Manon} Kaveh, K.; Manon, C. \textit{Toric vector bundles, valuations and tropical geometry}. Algebras and Representation Theory (2025). https://doi.org/10.1007/s10468-025-10336-7
\bibitem[KM24]{KM-TMB} Kaveh, K.; Manon, C. \textit{Tropical vector bundles and matroids}. arXiv: 2405.03576
\bibitem[KhM24]{KhM} Khan, B; Maclagan \textit{Tropical vector bundles}. arXiv:2405.03505 
\bibitem[Kh77]{Askold-toroidal} Khovanskii, A. G. \textit{Newton polyhedra, and toroidal varieties}. Funkcional. Anal. i Prilozen.11(1977), no.4, 56--64, 96.
\bibitem[Kh78]{Askold-genus} Khovanskii, A. G.
\textit{Newton polyhedra, and the genus of complete intersections}. Funktsional. Anal. i Prilozhen.12(1978), no.1, 51--61.
\bibitem[Kh88]{Askold-BZ} Khovanskii, A. G. {\it Algebra and mixed volumes}. Appendix in the book {\it Geometric inequalities} by Y.D. Burago and V.A. Zalgaller, Springer-Verlag, Berlin and New York. V. 285, 182--207, 1988. 
\bibitem[Kh16]{Askold-irr-comp} Khovanskii, A. G. {\it Newton polyhedra and irreducible components of complete intersections}. Izvestiya RAN, Ser. Matematika, V. 80, No 1, 2016, 281--304; translation in Izvestiya: Mathematics, V. 80, No 1, 2016, 263--304.
\bibitem[KhPu92]{KhPu} Khovanskii, A. G.; Pukhlikov, A. V. \textit{Finitely additive measures of virtual polyhedra}. Algebra i Analiz 4 (1992), no.2, 161--185; translation in St. Petersburg Math. J. 4 (1993), no.2, 337--356.
\bibitem[Kly89]{Klyachko} Klyachko, A. A. \textit{Equivariant bundles on toral varieties}. (Russian) Izv. Akad. Nauk SSSR Ser. Mat. 53 (1989), no. 5, 1001--1039, 1135; translation in Math. USSR-Izv. 35 (1990), no. 2, 337--375.
\bibitem[Payne08]{Payne-moduli} Payne, S. \textit{Moduli of toric vector bundles}. Compos. Math. 144 (2008), no. 5, 1199--1213.
\bibitem[Payne09]{Payne-cover} Payne, S. {\it Toric vector bundles, branched covers of fans, and the resolution property}, J. Alg. Geom. 18 (2009), 1--36.
\end{thebibliography}
\end{document}